\documentclass[11pt,reqno]{amsart}

\usepackage{fixltx2e} 
\usepackage[utf8]{inputenc} 
\usepackage{amssymb, latexsym, stmaryrd, amsthm, dsfont, amsfonts, amsbsy, amsmath, mathrsfs} 
\usepackage{mathtools} 
\usepackage{bm, bbm} 
\usepackage{enumerate} 
\usepackage{verbatim} 
\usepackage{url}   
\usepackage{lscape} 
\usepackage{centernot}
\usepackage[dvipsnames]{xcolor}
\usepackage[colorinlistoftodos]{todonotes} 
\usepackage{nicefrac} 
\usetikzlibrary{calc}

\usepackage{todonotes}

\usepackage{microtype,tikz,tikz-cd}  
\makeatletter 
\def\MT@register@subst@font{\MT@exp@one@n\MT@in@clist\font@name\MT@font@list
 \ifMT@inlist@\else\xdef\MT@font@list{\MT@font@list\font@name,}\fi}
\makeatother

\makeatletter 
\newcommand{\myitem}[1]{%
\item[(#1)]\protected@edef\@currentlabel{#1}%
}
\makeatother

\usepackage[pdftex,bookmarks,bookmarksnumbered,linktocpage,   
         colorlinks,linkcolor=blue,citecolor=blue]{hyperref}

\usepackage[capitalize,noabbrev]{cleveref} 

\newcommand{\bit}{\begin{itemize}}    
\newcommand{\eit}{\end{itemize}}
\newcommand{\ben}{\begin{enumerate}}
\newcommand{\een}{\end{enumerate}}

\newcommand{\benroman}{\ben[\normalfont (i)]}  
\let\eroman\een

\newcommand{\bde}{\begin{description}}
\newcommand{\ede}{\end{description}}
\let\oper=\mathbb                               

\newcommand{\III}{\oper{I}}                     
\newcommand{\SSS}{\oper{S}}                     
\newcommand{\VVV}{\oper{V}}                     

\theoremstyle{theorem}
\newtheorem{Theorem}{Theorem}[section]

\newtheorem{Proposition}[Theorem]{Proposition}

\newtheorem{Corollary}[Theorem]{Corollary}
\newtheorem{Claim}[Theorem]{Claim}

\theoremstyle{definition}
\newtheorem{Definition}[Theorem]{Definition}

\theoremstyle{remark}

\newtheorem{Remark}[Theorem]{Remark}

\crefname{Theorem}{Theorem}{Theorems}
\crefname{Proposition}{Proposition}{Propositions}
\crefname{Lemma}{Lemma}{Lemmas}
\crefname{Corollary}{Corollary}{Corollaries}
\crefname{Claim}{Claim}{Claims}
\crefname{Definition}{Definition}{Definitions}
\crefname{exa}{Example}{Examples}
\crefname{Remark}{Remark}{Remarks}
\crefname{Fact}{Fact}{Facts}
\crefname{exer}{Exercise}{Exercises}
\crefname{problem}{Problem}{Problems}

\AddToHook{env/Proposition/begin}{\crefalias{Theorem}{Proposition}}
\AddToHook{env/Lemma/begin}{\crefalias{Theorem}{Lemma}}
\AddToHook{env/Corollary/begin}{\crefalias{Theorem}{Corollary}}
\AddToHook{env/Claim/begin}{\crefalias{Theorem}{Claim}}
\AddToHook{env/Definition/begin}{\crefalias{Theorem}{Definition}}
\AddToHook{env/exa/begin}{\crefalias{Theorem}{exa}}
\AddToHook{env/Remark/begin}{\crefalias{Theorem}{Remark}}

\let\leq=\leqslant

\let\geq=\geqslant 

\usepackage[mathscr]{euscript}
 \let\mathscr\relax 
\usepackage[scr]{rsfso}

\renewcommand{\int}{\mathsf{int}\,}

\bmdefine{\A}{A} 
\bmdefine{\C}{C}                                
                              
\bmdefine{\2}{2}
\bmdefine{\B}{B}
\bmdefine{\D}{D}
\bmdefine{\E}{E}

\bmdefine{\Term}{T} 
\bmdefine{\Free}{F}
\bmdefine{\Fb}{F}

\usepackage{scalerel}

\newcommand{\K}{\mathsf{K}}
\newcommand{\M}{\mathsf{M}}

\newcommand{\HHH}{\mathbb{H}}
\newcommand{\PPP}{\mathbb{P}}
\newcommand{\QQQ}{\mathbb{Q}}

\newcommand{\PPU}{\mathbb{P}_{\!\textsc{\textup{u}}}^{}}

\newcommand{\ext}{\mathsf{ext}}
\newcommand{\extpp}{\mathsf{ext}_{\textsc{pp}}}

\renewcommand{\L}{\mathscr{L}}

\newcommand{\F}{\mathcal{F}}

\newcommand{\Con}{\mathsf{Con}}

\DeclareUnicodeCharacter{001B}{}



\begin{document}

\title{A completion of reduced commutative rings}

\author{Luca Carai, Miriam Kurtzhals, and Tommaso Moraschini}

\address{Luca Carai: Institute of Computer Science, Czech Academy of Sciences, Pod Vodárenskou věží 271/2, 182 00 Praha 8, Czechia}\email{carai@cs.cas.cz}

\address{Miriam Kurtzhals and Tommaso Moraschini: Departament de Filosofia, Facultat de Filosofia, Universitat de Barcelona (UB), Carrer Montalegre, $6$, $08001$ Barcelona, Spain}
\email{mkurtzku7@alumnes.ub.edu\\ tommaso.moraschini@ub.edu}

\date{\today}

\begin{abstract}
A commutative ring is \emph{reduced} when it can be embedded into a direct product of fields. While the category of reduced commutative rings plays a fundamental role in affine geometry, it exhibits several structural deficiencies:
it admits nonregular monomorphisms and epimorphisms, lacks amalgamation, and is not equationally axiomatizable. In this paper, we simultaneously repair these defects via a canonical completion in which all monomorphisms become regular. This completion is obtained by adjoining weak inverses and weak prime roots, turning the class of reduced commutative rings into a discriminator variety.
As a consequence, we obtain an explicit description of dominions in every class of reduced commutative rings containing all fields. This description is strikingly simple compared to that of dominions in the category of all commutative rings, as reflected in the Isbell–Mazet–Silver Zigzag Theorem.

\end{abstract}

\maketitle

\section{Introduction}

Reduced commutative rings play a central role in algebraic geometry, where they arise as coordinate rings of 
affine algebraic sets (see, e.g., \cite[p.~48]{milneAG}).
They are characterized by the absence of nonzero nilpotent elements or, equivalently, by their embeddability into direct products of fields (see, e.g.,  \cite[3.14 p.~23]{Wis91}).  
Despite its  ubiquity, the class of reduced commutative rings exhibits significant structural deficiencies: it admits nonregular monomorphisms and epimorphisms, lacks amalgamation, and is not even equationally axiomatizable.\footnote{In the category of reduced commutative rings, the existence of nonregular epimorphisms is equivalent to the existence of nonsurjective epimorphisms.} 
In this paper, we show that these defects can be simultaneously repaired by a  completion in which all monomorphisms become regular (see \cite[Sec.\ 11]{CKMIMPv2}).
Remarkably, this completion is obtained by adjoining  two simple kinds of unary operations, namely weak inverses and weak prime roots, and turns the class of reduced commutative rings into a discriminator variety.

Notably, this completion is canonical, in the sense that it coincides with the unique (up to term equivalence)  equational class in which monomorphisms are regular that, moreover, is isomorphic via the forgetful functor to a mono-reflective subcategory of the category of reduced commutative rings (see \cref{Cor : categorical description}). 
The uniqueness of this completion can also be understood internally, as follows. The operations that we add to reduced commutative rings are definable partial functions that, moreover, are preserved by homomorphisms (here called \emph{implicit operations}). Our completion is, up to term equivalence, the only expansion of reduced commutative rings with implicit operations in which every implicit operation becomes interpolated by a family of term functions (see \cite[Thm.\ 11.7]{CKMIMPv2}).

As a further consequence, we obtain a novel description of dominions in every class of reduced commutative rings containing all fields, such as the class of all integral domains, in terms of weak inverses and weak prime roots. The simplicity of this description contrasts sharply with the one of dominions in the ambient category of all commutative rings given by the Isbell–Mazet–Silver Zigzag Theorem (see \cite[Thm.~1.1]{IsbEpiIV}), which combines results from  \cite[p.\ 2-07]{Mazet-epi-seminar} and \cite[Prop.\ 1.1]{MR217114}.

\subsection*{Dominions and the regularity of monos and epis} 

By an \emph{algebra} $\A$ we understand a set $A$ endowed with a family of finitary operations on $A$. Familiar examples of algebras include monoids, groups, rings, and Boolean algebras. Two algebras are said to be \emph{similar} when they share the same language (see, e.g., \cite[Chap.\ 1.1]{Ber11} and \cite[Chap.\ 1.3]{ModCK}). Every class $\K$ of similar algebras can be viewed as a category whose objects are the members of $\K$ and whose arrows are the homomorphisms between them. We recall that an arrow in $\K$ is said to be a  \emph{monomorphism} (resp.\  \emph{epimorphism})  in $\K$ when it is left (resp.\ right) cancellable. 
A monomorphism (resp.\ epimorphism) in $\K$ is \emph{regular} when it is the equalizer (resp.\ coequalizer) of a parallel pair of arrows in $\K$ 
(see, e.g., \cite[Defs.\ 7.56 \& 7.71]{AHS06}).

While reduced commutative rings cannot be axiomatized equationally, they can still be axiomatized by implications between finite sets of equations. Classes of algebras of this kind are known as \emph{quasivarieties} and were introduced by Maltsev  (see  \cite{MR349384}). We recall that quasivarieties form bicomplete categories (see, e.g., \cite[Prop.~9.4.8 \& Thm.~9.4.14]{Ber15}) in which
 injective arrows coincide with monomorphisms and surjective arrows with regular epimorphisms  
(see, e.g., \cite[p.~222]{McK96} and \cite[Rem.~5.13(2)]{AR94}).

  In quasivarieties, the demand that all monomorphisms (resp.\ epimorphisms) be regular can be formulated in terms of Isbell's dominions (see \cite{Isb65}), as we proceed to recall. Let $\K$ be a class of similar algebras. Given a subalgebra $\A$  of a member $\B$ of $\K$ (in symbols $\A \leq \B \in \K$), the \emph{dominion} of $\A$ in $\B$ relative to $\K$ is the set
\begin{align*}
\mathsf{d}_\K(\A, \B) = \{ & b \in B : g(b) = h(b) \text{ for every pair of homomorphisms} \\
&g, h \colon \B \to \C  \text{ with }\C \in \K \text{ such that }g{\upharpoonright}_A = h{\upharpoonright}_A\}.
\end{align*}
When $\K$ is a quasivariety, the following holds (see, e.g., \cite[Prop.~6.1]{SurvKissal} and \cite[p.~1]{Isb65}):
\begin{align*}
    \text{every monomorphism in $\K$ is regular} &\iff \text{$\mathsf{d}_\K(\A, \B) = A$ for all $\A \leq \B \in \K$};\\
    \text{every epimorphism in $\K$ is regular} &\iff \text{$\mathsf{d}_\K(\A, \B) \ne B$ for all $\A < \B \in \K$},
\end{align*}
where $\A < \B$ means that $\A$ is a proper subalgebra of $\B$. Consequently, in the context of quasivarieties, the regularity of monomorphisms implies that of epimorphisms. 

\subsection*{Regularizing monomorphisms} As we mentioned, our aim is to complete the category of reduced commutative rings so that every monomorphism becomes regular. Since the requirement that all monomorphisms be regular in a quasivariety 
$\K$ can be expressed in terms of dominions, solving this problem necessitates a detailed analysis of dominions, which is facilitated by the following concept (see \cite[Sec.\ 3]{CKMIMPv2}).

An $n$-ary \emph{operation} of $\mathsf{K}$ is a family $f = \langle f^\A : \A \in \mathsf{K} \rangle$, where each $f^\A \colon \mathsf{dom}(f^\A) \to A$ is a partial $n$-ary function on $A$ with domain $\mathsf{dom}(f^\A) \subseteq A^n$ that is  globally preserved by the homomorphisms between members of $\mathsf{K}$.
The latter means that for every homomorphism $h \colon \A \to \B$ with $\A, \B \in \K$ and $\langle a_1, \dots, a_n \rangle \in \mathsf{dom}(f^\A)$ we have
\[
\langle h(a_1), \dots, h(a_n) \rangle \in \mathsf{dom}(f^\B) \, \, \text{ and } \, \, h(f^\A(a_1, \dots, a_n)) = f^\B(h(a_1), \dots, h(a_n)).
\]
 An $n$-ary operation $f$  of $\K$ is said to be \emph{implicit} when it is defined by some first order formula $\varphi(x_1, \dots, x_n, y)$, in the sense that for all $\A \in \K$ and $a_1, \dots, a_n, b \in A$,
 \[
\A \vDash \varphi(a_1, \dots, a_n, b) \iff \langle a_1, \dots, a_n \rangle \in \mathsf{dom}(f^\A) \text{ and }f^\A(a_1, \dots, a_n) = b.
 \]
 For instance, “taking multiplicative inverses” is an implicit operation of commutative rings because it is defined by the equation $xy \thickapprox 1$ and ring homomorphisms preserve  multiplicative inverses, when they exist.
Implicit operations and dominions are connected as follows (see \cite[Thm.~1]{Bacsich47} and \cite[Thm.~3.2]{Camper18jsl}). When $\K$ is a quasivariety and $\A \leq \B \in \K$, the dominion $\mathsf{d}_\K(\A, \B)$ is the result of closing $A$ in $B$ under all the implicit operations of $\K$. Formally, $\mathsf{d}_\K(\A, \B)$ consists of the elements $b \in B$ for which there exist an implicit operation $f$ of $\K$ and $\langle a_1, \dots, a_n \rangle \in \mathsf{dom}(f^\B) \cap A^n$ such that $b = f^\B(a_1, \dots, a_n)$.

This suggests a canonical method for completing a quasivariety $\K$ so that every monomorphism becomes regular. The overall strategy consists of adding enough implicit operations to $\K$ in such a way that the resulting expansion $\mathsf{C}$ is still a quasivariety and, moreover, satisfies the following closure property: if $\A \leq \B \in \mathsf{C}$, then $A$ is closed in $B$ under the implicit operations of $\mathsf{C}$. Consequently, for all $\A \leq \B \in \mathsf{C}$ we have $\mathsf{d}_\mathsf{C}(\A, \B) = A$, which ensures that all monomorphisms in $\mathsf{C}$ are regular. 

The result of such an expansion of $\K$ has been called a \emph{Beth companion} of $\K$ in \cite[Sec.\ 11]{CKMIMPv2}. For instance, the class of Abelian groups is a Beth companion of the quasivariety of cancellative commutative monoids (see \cite[Thm.\ 11.9(i)]{CKMIMPv2}), obtained by adding the implicit operation of “taking  inverses”.
While a quasivariety may lack a Beth companion  (see, e.g., \cite[Thm.~14.17]{CKMIMPv2} and \cite[Thm.\ 6.1]{CKMMON}), 
Beth companions are essentially unique when they exist because all Beth companions of a quasivariety $\K$ are term equivalent (see \cite[Thm.\ 11.7]{CKMIMPv2}).  In a slogan, Beth companions are the natural completions that regularize monomorphisms for quasivarieties.

\subsection*{The main result}  

Let $\mathsf{RCR}$ be the quasivariety of reduced commutative rings. In this paper, we show that  $\mathsf{RCR}$ admits a Beth companion, obtained by adjoining the implicit operations of “taking weak inverses” and “taking weak prime roots” (see \cref{Thm : main}). In addition, this completion of $\mathsf{RCR}$ is canonical, in the sense that it coincides with the unique (up to term equivalence) quasivariety in which monomorphisms are regular that, moreover, is isomorphic via the forgetful functor to a mono-reflective subcategory of $\mathsf{RCR}$ (see \cref{Cor : categorical description}). 
Beyond ensuring the regularity of monomorphisms (and  hence of epimorphisms), it also repairs the failure of the amalgamation property and yields a discriminator variety (see \cref{Thm : ICM has AP} and \cref{Rem : discriminator}).\
Thus, a minimal expansion transforms $\mathsf{RCR}$ into a remarkably structured class.

In order to describe concretely the Beth companion of $\mathsf{RCR}$, we recall that the characteristic of a field is either zero or prime (see, e.g., \cite[p.\ 30]{FLField}), and that every element $a$ of a field of prime characteristic $p$ has at most one $p$-th root that, when existing, will be denoted by $\sqrt[p]{a}$ (see, e.g., 
\cite[F14 p.\ 71]{FLField}). 
A field $\A$ is said to be \emph{perfect}
when one of the following conditions holds:
\benroman
\item $\A$ has characteristic $0$;
\item $\A$ has prime characteristic $p$ and contains $\sqrt[p]{a}$ for every $a \in A$.
\eroman
Besides the fields of characteristic zero, examples of perfect fields include all finite fields and all algebraically closed fields 
(see \cref{Prop : ACF are weakly rooted}).
Given a prime $p$, the \emph{weak $p$-root} of an element $a$ of a perfect field $\A$ is 
\[
 r_p(a) = \begin{cases*}
                    \sqrt[p]{a} & if  $\A$ has characteristic $p$;  \\
                     0 & otherwise.
                 \end{cases*}
 \]
In addition, the  \emph{weak inverse} of an element $a$ of a field $\A$ is
\[
 a^* = \begin{cases*}
                    a^{-1} & if  $a \ne 0$;  \\
                     0 & if $a = 0$,
                 \end{cases*}
 \]
 where $a^{-1}$ is the multiplicative inverse of $a$. Lastly, an \emph{expanded perfect field} is an algebra $\langle A; +, \cdot, -, 0, 1, (\,)^*, \{ r_p : \text{$p$ is prime}\}\rangle$, where $\langle A; +, \cdot, -, 0, 1\rangle$ is a perfect field and $(\,)^*$ and $r_p$ are the unary operations of taking weak inverses and weak $p$-roots, respectively. 

Our main result states that the Beth companion of $\mathsf{RCR}$  is the class of all the algebras that can be embedded into a direct product of expanded perfect fields.
We term these algebras \emph{perfect meadows} and show that they can be axiomatized by the axioms of commutative rings together with the equations 
\[
x \thickapprox  x^2 x^*, \quad x \thickapprox x^{**}, \quad  (r_p(x))^p \thickapprox  (1-p^*p)x
\]
for every prime $p$, where we denote by $p$ the result of summing $p$-times the unit $1$ (see \cref{Thm : ICM axioms}). 

\subsubsection*{Related work}  The fact that perfect meadows form the Beth companion of $\mathsf{RCR}$ was subsequently reproved using categorical methods in \cite{DR26}. The original proof, presented here, is purely algebraic and yields an equational axiomatization of perfect meadows.

\subsection*{Dominions}  

As a consequence of our main result, we obtain a novel characterization of dominions in every class $\K \subseteq \mathsf{RCR}$ containing all fields (see \cref{Thm : dominions}). In addition to $\mathsf{RCR}$ and the class of all fields, these comprise, for example, the class of all integral domains. Besides its wide applicability, an attractive feature of our characterization lies in its simplicity, which is in stark contrast with the more involved description of dominions in the category of all commutative rings given by the Isbell-Mazet-Silver Zigzag Theorem  (see \cite[Thm.\ 2.9]{Isb65}, \cite[Thm.~1.1]{IsbEpiIV}, \cite[p.\ 2-07]{Mazet-epi-seminar}, and \cite[Prop.\ 1.1]{MR217114}).\footnote{We remark that our result does not follow from the Isbell-Mazet-Silver Zigzag Theorem because, when $\A \leq \B \in \K \subseteq \K'$ for a class of similar algebras $\K'$, the equality $\mathsf{d}_\K(\A, \B) = \mathsf{d}_{\K'}(\A, \B)$ need not hold.}

More precisely, for all $\A \in \mathsf{RCR}$ and prime ideals $I$ of $\A$, let $\mathsf{acl}(\mathsf{frac}(\A / I))$ be the algebraic closure of the fraction field of $\A / I$, and $\mathsf{icf}_I(\A)$  the unique expanded perfect field whose field reduct is $\mathsf{acl}(\mathsf{frac}(\A / I))$. 
We show that for every class $\K \subseteq \mathsf{RCR}$ that contains all fields and for all $\A \leq \B \in \K$,
\begin{align*}
    \mathsf{d}_\K(\A, \B) = \{ & b \in B : \text{for every pair $I$ and $J$ of prime ideals of $\B$,  }\\
&\langle b + I, b + J \rangle \text{ belongs to the subalgebra of }\mathsf{icf}_I(\B) \times \mathsf{icf}_J(\B)\\
& \text{generated by }\{ \langle a + I, a + J \rangle : a \in A \} \}.
\end{align*}
This description becomes particularly simple when $\K$ is the class $\mathsf{F}$ of all fields, as for all $\A \leq \B \in \mathsf{F}$ the dominion $\mathsf{d}_\mathsf{F}(\A, \B)$
 is the least subfield of the algebraic closure of $\B$ that contains $A$ and, when $\B$ has characteristic $p$, is also closed under $p$-roots (see \cref{Cor : dominions in fields}).

\section{Field theory}

\subsection{Algebras} We recall  that an \emph{algebraic language} is a family $\L$ of function symbols together with a map $\mathsf{ar} \colon \L \to \mathbb{N}$ that associates an arity with every member of $\L$. Then, an $\L$-\emph{algebra} is a structure $\A = \langle A; \{ f^\A : f \in \L \} \rangle$, where $A$ is set  and $f^\A$ is a function on $A$ of arity $\mathsf{ar}(f)$ for every $f \in \L$ (see, e.g., \cite[Sec.\ 1.1]{Ber11} and \cite[Sec.\ 1.3]{ModCK}).\ Two algebras are said to be \emph{similar} when they have the same language. A class of similar algebras is  called \emph{elementary} when it can be axiomatized by first order formulas.

A \emph{term} of a class of similar algebras $\K$ is a formal expression obtained by applying the function symbols of $\K$ to the set of variables. For instance, $x + ( y \cdot -z)$ is a ring term. With every term $t(x_1, \dots, x_n)$ of $\K$ and $\A \in \K$ we associate a map $t^\A \colon A^n \to A$ that sends a tuple $\langle a_1, \dots, a_n \rangle \in A^n$ to the result of applying the interpretation of $t$ in $\A$ to the elements $a_1, \dots, a_n$. For instance, if $t(x, y, z) = x + (y \cdot - z)$ and $\A$ is a ring, then $t^\A \colon A^3 \to A$ is the map defined as $t^\A(a, b, c) = a + (b \cdot -c)$ for all $a, b, c \in A$ (see, e.g., \cite[Sec.\ 4.3]{Ber11}).

 Let $\A$ and $\B$ be a pair of similar algebras. We write $\A \leq \B$ to indicate that $\A$ is a subalgebra of $\B$. Moreover, when $h \colon \A \to \B$ is a homomorphism, we denote the subalgebra of $\B$ with universe $h[A]$ by $h[\A]$. The subalgebra of $\A$ generated by some $X \subseteq A$ will be denoted by $\mathsf{Sg}^{\A}(X)$.
We recall that 
\[
\mathsf{Sg}^\A(X) = \{ t^\A(a_1, \dots, a_n) :  t(x_1, \dots, x_n) \text{ is a term of $\A$ and }a_1, \dots, a_n \in A \}
\]
(see, e.g., \cite[Thm.\ 1.14]{Ber11}).\
 Lastly, let $\III, \HHH, \SSS, \PPP$, and $\PPU$ be the class operators of closure under isomorphisms, homomorphic images, subalgebras, direct products, and ultraproducts, respectively. 

\subsection{Fields and integral domains}
 Throughout this work, rings will be considered as algebras in the language $\mathscr{L} = \{ +, \cdot, -, 0, 1  \}$. Consequently, all rings will be assumed to be unital. When it exists, we denote the multiplicative inverse of an element $a$ of a ring by $a^{-1}$.
\begin{Definition}
A nontrivial commutative ring $\A$ is said to be:
\benroman
\item a \emph{field} when every nonzero element of $A$ has a multiplicative inverse;
\item an \emph{integral domain} when for all $a, b \in A$ such that $ab = 0$ we have $0 \in \{ a, b\}$ or, equivalently, when   
$\A \leq \B$ for some field $\B$ 
(see, e.g., \cite[F7 p.\ 28]{FLField}).
\eroman
The classes of all fields and of all integral domains will be denoted, respectively, by $\mathsf{F}$  and $\mathsf{ID}$.
\end{Definition}

\begin{Definition}
A commutative ring $\A$ is said to be \emph{reduced} when  $a^2 = 0$ implies $a= 0$ for every $a \in A$ or, equivalently, when $\A \in \III\SSS\PPP(\mathsf{F})$ (see, e.g., \cite[3.14 p.~23]{Wis91}). We denote the class of reduced commutative rings by $\mathsf{RCR}$.
\end{Definition}

By definition we have
\[
\mathsf{ID} = \SSS(\mathsf{F}) \, \, \text{ and }\, \, \mathsf{RCR} = \III\SSS\PPP(\mathsf{F}).
\]

We recall that every ideal $I$ of a commutative ring $\A$ induces a quotient ring $\A / I$. 

\begin{Proposition}[\protect{\cite[Def.\ 9, p.~41]{FLField}}]\label{Prop : prime ideals and integral domains}
 Let $I$ be an ideal of a commutative ring $\A$. Then $I$ is prime if and only if $\A / I$ is an integral domain.
\end{Proposition}

We also recall that every integral domain $\A$ can be embedded into its field of fractions $\mathsf{frac}(\A)$. To simplify the notation, from now on we will assume that $\A \leq \mathsf{frac}(\A)$.

\begin{Proposition}[\protect{\cite[F7 p.\ 28]{FLField}}]\label{Prop : embedding fields of fractions}
Let $\A \in \mathsf{ID}$ and $\B \in \mathsf{F}$. For every embedding $h \colon \A \to \B$ there exists an embedding $g \colon \mathsf{frac}(\A) \to \B$ such that $ g{\upharpoonright}_A = h$.
\end{Proposition}

For every $n \in \mathbb{N}$ we define recursively a term $nx$ in the language of rings by setting
   \[
   0x = 0\, \, \text{ and }\, \,  (n+1)x = nx+x.
   \]
\begin{Definition}
The \emph{characteristic} of a field $\A$ is the smallest $n \in \mathbb{Z}^+$ such that in $\A$ we have $n1 = 0$ if such an $n$ exists and $0$ otherwise.
\end{Definition}

\begin{Proposition}[\protect{\cite[p.\ 30]{FLField}}]\label{Prop : characteristic field prime or zero}
    The characteristic of a field is either prime or zero.
\end{Proposition}

\subsection{Field extensions}

We denote the ring of polynomials in  variables $x_1, \dots, x_n$  with coefficients in a field $\A$ by  $\A[x_1, \dots, x_n]$ and recall that $\A[x_1, \dots, x_n]$  is always an integral domain (see, e.g., \cite[Remark p.\ 149]{KnaAlg}).

\begin{Definition}
    A \emph{field extension} consists of a pair of fields $\A$ and $\B$ such that $\A \leq \B$. Given a field extension $\A \leq \B$, an element $b \in B$ is said to be:
    \benroman
\item \emph{algebraic} over $\A$ when there exists a nonzero polynomial $p(x) \in \A[x]$ with coefficients in $\A$ such that $p(b) = 0$;
\item \emph{transcendental} over $\A$ when it is not algebraic over $\A$.
    \eroman
\end{Definition}

Let $\A \leq \B$ be a field extension and $b \in B$ algebraic over $\A$. Then there exists a unique monic polynomial $\mu_b(x) \in \A[x]$, known as the \emph{minimal polynomial} of $b$ over $\A$, with $\mu_b(b) = 0$ that is of minimal degree among the members of $\A[x]$ with root $b$ (see, e.g., \cite[p.17]{FLField}). It is well known that $\mu_b$ is irreducible (see, e.g., \cite[F6\ p.~26]{FLField}). Lastly, given a field extension $\A \leq \B$ and  $b_1, \dots, b_n \in B$,  we denote the smallest subfield of $\B$ containing  $A \cup \{ b_1, \dots, b_n \}$ by $\A(b_1, \dots, b_n)$  (see, e.g., \cite[p.~6]{FLField}). Moreover, we write  $\A(x_1, \dots, x_n)$ as a shorthand for $\mathsf{frac}(\A[x_1, \dots, x_n])$.

\begin{Proposition} \label{Prop : special field ext homs}
The following conditions hold for all field extensions $\A \leq \B$ and $b \in B$:
\benroman
\item \label{item : field ext hom alg} if $b$ is algebraic over $\A$ and $c$ is a root of $\mu_b$ in some field $\C$ with $\A \leq \C$, there exists an isomorphism $h \colon \A(b) \to \A(c)$ such that
\[
h(b) = c \, \, \text{ and }\, \, h(a) = a \, \, \text{ for every } a \in A;
\]
\item \label{item : field ext hom trans} if $b$ is transcendental over $\A$, there exists an isomorphism $h \colon \A(b) \to \A(x)$ such that 
\[
h(b) = x \, \, \text{ and }\, \, h(a) = a \, \, \text{ for every } a \in A.
\]
\eroman
\end{Proposition}

\begin{proof}
See, e.g., \cite[F1\ p.~55]{FLField} and 
    \cite[F9\ p.~30]{FLField}. 
\end{proof}

We will make use of the following concept.

\begin{Definition}
A field $\A$ is said to be \emph{algebraically closed} when every nonconstant polynomial in $\A[x]$ has a root in $\A$. Given a field extension $\A \leq \B$, we say that
\benroman
\item $\A \leq \B$ is  \emph{algebraic} when every member of $\B$ is algebraic over $\A$;
  \item $\B$ is  an \emph{algebraic closure} of $\A$ when it is algebraically closed and $\A \leq \B$ is algebraic.
\eroman
\end{Definition}

\begin{Theorem}[\protect{\cite[Thm.\ 2 p.\ 57]{FLField}}]\label{Thm : existence of ACL}
Every field $\A$ has an algebraic closure that, moreover, is unique up to isomorphism, in the sense that if $\B$ and $\C$ are algebraic closures of $\A$, there exists an isomorphism $h \colon \B \to \C$ that fixes every member of $\A$.
\end{Theorem}

Consequently, we may talk about \emph{the} algebraic closure of a field $\A$, which we denote by $\mathsf{acl}(\A)$. We will rely on the next property of algebraic closures (see, e.g., \cite[Thm.\ 9.23 p.\ 463]{KnaAlg}), for which we recall that every homomorphism between fields is an embedding (see, e.g., \cite[F15 p.\ 42]{FLField}).

\begin{Proposition}\label{Prop : extensing acl embeddings}
Let $\A, \B \in \mathsf{F}$. For every embedding $h \colon \A \to \B$ there exists an embedding $g \colon \mathsf{acl}(\A) \to \mathsf{acl}(\B)$ such that $h = g{\upharpoonright}_A$.
\end{Proposition}

 Given a ring $\A$, the degree of a polynomial $p \in \A[x]$ will be denoted by $\mathsf{deg}(p)$. We will make use of separable extensions and perfect fields, which are usually defined in terms of splitting fields (see, e.g., \cite[Chap.\ 7]{FLField}). However, for the present purpose, it is convenient to adopt the following definition,
which is equivalent because every splitting field of a field $\A$ embeds into the algebraic closure of $\A$ (see, e.g., \cite[proof of F4 p.\ 61]{FLField}).

    \begin{Definition} 
Given a field $\A$, we say that
\benroman
\item an algebraic field extension $\A \leq \B$ is \emph{separable} when the minimal polynomial $\mu_b$ of every $b \in B$ over $\A$ has $\mathsf{deg}(\mu_b)$ distinct roots in $\mathsf{acl}(\A)$;
   \item $\A$ is  \emph{perfect} when every algebraic extension of $\A$ is separable. 
   \eroman
    \end{Definition}

\section{Perfect meadows}

\subsection{Field expansions}  

Recall from \cref{Prop : characteristic field prime or zero} that the characteristic of a field is either zero or prime. An element $a$ of a field of prime characteristic $p$ has at most one $p$-th root that, when it exists, will be denoted by $\sqrt[p]{a}$ (see, e.g., 
\cite[F14 p.\ 71]{FLField}).

\begin{Definition}
A field $\A$ is \emph{perfect} when one of the following conditions holds:
\benroman
\item $\A$ has characteristic $0$;
\item $\A$ has prime characteristic $p$ and contains $\sqrt[p]{a}$ for every $a \in A$.
\eroman
                The class of perfect fields will be denoted by $\mathsf{PF}$.
\end{Definition}

\begin{Proposition}\label{Prop : ACF are weakly rooted}
Every finite or algebraically closed field is perfect.
\end{Proposition}

\begin{proof}
From \cite[Remark, p.\ 71]{FLField} it follows that every finite field is perfect.
   Then let $\A$ be an algebraically closed field.\  If $\A$ has characteristic zero, we are done. Then we consider the case where $\A$ has prime characteristic $p$. We need to prove that $\sqrt[p]{a}$ exists for every $a \in A$. To this end, consider $a \in A$. As $\A$ is algebraically closed, the polynomial $x^p - a$ has a root $b$ in $\A$, i.e., $b^p = a$. Hence, we conclude that $b = \sqrt[p]{a}$. 
\end{proof}

From \cref{Thm : existence of ACL} and \cref{Prop : ACF are weakly rooted} we deduce the following.

\begin{Corollary}\label{Cor : every field extend to a rooted one}
Every field can be extended to a perfect one.
\end{Corollary}

We say that an algebra $\A$ is a \emph{reduct} of an algebra $\B$ when there exists a sublanguage $\L$ of the language of $\B$ such that $\A = \langle B; \{ f^\B : f \in \L \} \rangle$. In this case, we also say that $\B$ is an \emph{expansion} of $\A$. A \emph{field expansion} is an expansion of a field. The \emph{characteristic} of a field expansion $\A$ is the characteristic of the field reduct of $\A$.

We will focus on field expansions obtained adding the following operations. Given a field $\A$, let $(\,)^* \colon A \to A$ be the map defined for every $a \in A$ as
\begin{equation}\label{Eq : weak inverses def}
 a^* = \begin{cases*}
                    a^{-1} & if  $a \ne 0$;  \\
                     0 & if $a = 0$.
                 \end{cases*}
\end{equation}
The element $a^*$ is said to be a \emph{weak inverse} of $a$. Moreover, if $\A$ is a perfect field and $p$ prime, let $r_p \colon A \to A$ be the map defined for every $a \in A$ as
\begin{equation}\label{Eq : weak prime roots}
 r_p(a) = \begin{cases*}
                    \sqrt[p]{a} & if  $\A$ has characteristic $p$;  \\
                     0 & otherwise.
                 \end{cases*}
\end{equation}
We call $r_p(a)$ the \emph{weak $p$-root} of $a$. When convenient, we write $a^{*\A}$ and $r_p^\A(a)$ (as opposed to $a^*$ and $r_p(a)$ only) to stress that these elements are computed in $\A$.

\begin{Definition}
An algebra $\langle A; +, \cdot, -, 0, 1, (\,)^* \rangle$ is said to be a \emph{zero-totalized field} when $\langle A; +, \cdot, -, 0, 1 \rangle$ is a field and $(\,)^*$ the unary operation in \eqref{Eq : weak inverses def}.
                 The class of zero-totalized fields will be denoted by $\mathsf{ZTF}$. A zero-totalized field is \emph{perfect} when its field reduct is perfect.  
\end{Definition}

In the context of zero-totalized perfect fields, the operations $r_p$ 
can be defined equationally, as we proceed to illustrate. With every prime $p$ we associate the equation
\[
\mathsf{root}_p(x, y) = y^p \thickapprox  (1-p^*p)x
\]
in the language of $\mathsf{ZTF}$. Observe that the expression $1-p^*p$ equals $1$ in zero-totalized fields of characteristic $p$, and equals $0$ otherwise.
\begin{Proposition}\label{Prop : graph}
Let $\A \in \mathsf{ZTF}$ be perfect.\ For every
prime $p$ and $a, b \in A$,
\[
\A \vDash \mathsf{root}_p(a, b) \iff b = r_p(a).
\]
\end{Proposition}

\begin{proof}
We have two cases depending on whether $\A$ has characteristic $p$ or not (see \cref{Prop : characteristic field prime or zero}). First, suppose that $\A$ has characteristic $p$. We have
\[
\A \vDash  \mathsf{root}_p(a, b) \iff b^p = (1-p^*p)a
\iff b^p = a \iff b = \sqrt[p]{a} \iff b = r_p(a),  
\]
where the first equivalence above holds by the definition of the equation $\mathsf{root}_p$,  the second by the assumption that $\A$ has characteristic $p$, 
the third is straightforward, and the fourth holds by the definition of a weak $p$-root.

Lastly, we consider the case where $\A$ does not have characteristic $p$.\ We have
\[
\A \vDash \mathsf{root}_p(a, b) \iff b^p = (1-p^*p)a
\iff b^p = 0 \iff b = 0,  
\]
where the first equivalence above holds by the definition of the equation $\mathsf{root}_p$, the second because  $\A$ does not have characteristic $p$, 
and the third because fields are integral domains.
\end{proof}

For the present purpose, the most important field expansions are the following.

\begin{Definition}
An \emph{expanded perfect field} is an algebra  $\langle A; +, \cdot, -, 0, 1, (\,)^*, \{ r_p : \text{$p$ is prime}\}\rangle$, where $\langle A; +, \cdot, -, 0, 1, (\,)^*\rangle$ is a zero-totalized perfect field and the $r_p$'s are the unary operations in~\eqref{Eq : weak prime roots}.
                  The class of expanded perfect fields will be denoted by $\mathsf{EPF}$.

\end{Definition}

\begin{Definition}
We denote by $\A^+$ the field expansion of a perfect field $\A$ obtained by adding the operation $(\,)^*$ as well as all the $r_p$'s for $p$ prime.
\end{Definition}

From \cref{Prop : ACF are weakly rooted} we deduce the following.

\begin{Proposition}\label{Prop : acl(A) in ICF}
Let $\A$ be a field. Then  $\mathsf{acl}(\A)^+$ is an expanded perfect field.
\end{Proposition}

The next observation will be needed later on.

\begin{Proposition}\label{Prop : ICF closed under S and Pu}
We have that $\mathsf{PF}$ and $\mathsf{EPF}$ are elementary classes. Moreover, $\mathsf{PF} = \PPU(\mathsf{PF})$ and $\mathsf{EPF} = \III\SSS\PPU(\mathsf{EPF})$. 
\end{Proposition}

\begin{proof}
The class $\mathsf{PF}$ can be axiomatized by adding the following first order formulas to the equational axioms of commutative rings:
\[
\forall x (x \not \thickapprox 0 \to \exists y (xy \thickapprox 1)) \, \, \text{ and }\, \, 
p \thickapprox 0 \to \forall x \exists y (x \thickapprox  y^p) \text{ for all primes }p.
\]
Therefore, $\mathsf{PF}$ is an elementary class. Since every elementary class is closed under $\PPU$ (see, e.g., \cite[Thm.\ V.2.16]{BuSa00}), it follows that so is $\mathsf{PF}$.  As the inclusion $\mathsf{PF} \subseteq \PPU(\mathsf{PF})$ is straightforward, we conclude that $\mathsf{PF} = \PPU(\mathsf{PF})$.

The class $\mathsf{EPF}$ can be axiomatized by adding the following formulas to the equational axioms of commutative rings: for all prime $p$,
\begin{align*}
    \forall x (x \thickapprox 0 \to x^* \thickapprox 0), \qquad &\forall x (x \not \thickapprox 0 \to x^* x \thickapprox 1), \\
\forall x (p \thickapprox 0 \to x \thickapprox (r_p(x))^p), \qquad &\forall x  (p \not \thickapprox 0 \to r_p(x) \thickapprox 0).    
\end{align*}
Thus, $\mathsf{EPF}$ is an elementary class.
Since every class of algebras that can be axiomatized by first order formulas of the form $\forall x_1, \dots, x_n \varphi$ with $\varphi$ quantifier-free is closed under $\III, \SSS$, and $\PPU$ (see, e.g., \cite[Thm.\ V.2.20]{BuSa00}), so is $\mathsf{EPF}$. 
As the inclusion $\mathsf{EPF} \subseteq \III\SSS\PPU(\mathsf{EPF})$ is straightforward, we conclude that $\mathsf{EPF} = \III\SSS\PPU(\mathsf{EPF})$. 
\end{proof}

\subsection{Perfect meadows}\label{Sec : Q not V} We begin by recalling the following concepts.
\begin{Definition}
A class of similar algebras is 
\benroman
\item a \emph{quasivariety} when it is closed under $\III, \SSS, \PPP$, and $\PPU$ or, equivalently, when it can be axiomatized by a set of \emph{quasiequations}, that is, formulas of the form $\bigsqcap \Phi \to \varphi$, where $\Phi \cup \{ \varphi \}$ is a finite set of equations and $\bigsqcap$ the conjunction symbol (see, e.g.,  \cite[Thm.\ V.2.25]{BuSa00} due to Maltsev);
\item a \emph{variety} when it is closed under $\HHH, \SSS$, and $\PPP$ or, equivalently, when it can be axiomatized by a set of equations (see, e.g., \cite[Thm.\ II.11.9]{BuSa00}  due to Birkhoff).
\eroman
\end{Definition}
Observe that every variety is a quasivariety because every equation $\varphi$ is equivalent to the quasiequation $\bigsqcap \emptyset \to \varphi$. However, the converse does not need to hold in general: for example, $\mathsf{RCR}$ is a quasivariety  axiomatized by the equations of commutative rings together with the quasiequation $x^2 \thickapprox 0 \to x \thickapprox 0$, but it is not a variety because it is not closed under $\HHH$. To prove the latter, observe that  the ring of integers $\mathbb{Z}$ is reduced and $\mathbb{Z}/4\mathbb{Z} \in \HHH(\mathbb{Z})$ is not, for $2^2 = 0$ in $\mathbb{Z}/4\mathbb{Z}$. 

For every class of similar algebras $\K$ there exist the least quasivariety and the least variety containing $\K$, which we denote by $\QQQ(\K)$ and $\VVV(\K)$, respectively. By theorems of Maltsev and Tarski, respectively, the latter can be described as follows (see, e.g.,  \cite[Thms.\ V.2.25 and II.9.5]{BuSa00}).

\begin{Theorem}\label{Thm : Q(K) and V(K)}
Let $\K$ be a class of similar algebras. Then 
\[
\QQQ(\K) = \III\SSS\PPP\PPU(\K) \, \, \text{ and }\, \, \VVV(\K) = \HHH\SSS\PPP(\K).
\]
\end{Theorem}

The following structures will play a fundamental role in this paper.

\begin{Definition}
We write $\mathsf{PM}$ as a shorthand for $\III\SSS\PPP(\mathsf{EPF})$. A \emph{perfect meadow} is a member of $\mathsf{PM}$.
\end{Definition}

With every prime $p$ we associate the equation
   \[
  \mathsf{root}_p(x,r_p(x)) = (r_p(x))^p \thickapprox  (1-p^*p)x
   \]
in the language of $\mathsf{PM}$. Our first goal will be to axiomatize $\mathsf{PM}$. For this, we recall that the class of commutative rings can be axiomatized by equations.

\begin{Theorem}\label{Thm : ICM axioms}
    $\mathsf{PM}$ is a variety axiomatized by the equations
    \[
    x \thickapprox  x^2 x^*, \quad x \thickapprox x^{**}, \quad   \mathsf{root}_p(x,r_p(x))
    \]
for each prime $p$, together with    the axioms of commutative rings.
\end{Theorem}

The next concept will be instrumental for this purpose.

\begin{Definition}
We write $\mathsf{M}$ as a shorthand for $\III\SSS\PPP(\mathsf{ZTF})$. A \emph{meadow} is a member of $\mathsf{M}$ (see \cite[Sec.\ 3.2]{MedBHT}).
\end{Definition}

\begin{Theorem}[\protect{\cite{MedBHT}}]\label{Thm : meadows axioms}
$\mathsf{M}$ is a variety axiomatized by the equations
    \[
    x \thickapprox x^2 x^* \, \, \text{ and }\, \, x \thickapprox x^{**},
    \]
 together with the axioms of commutative rings.
\end{Theorem}

We are now ready to prove \cref{Thm : ICM axioms}.

\begin{proof}
We begin by proving that $\mathsf{PM}$ satisfies the axioms in the statement. Recall that $\mathsf{PM} = \III\SSS\PPP(\mathsf{EPF})$ by definition
and observe that $\mathsf{EPF}$ satisfies the (equational) axioms of commutative rings as well as the equations $x \thickapprox  x^2x^*$ and $x \thickapprox x^{**}$. Therefore, so does $\mathsf{PM}$ because the validity of equations is preserved by $\III, \SSS$, and $\PPP$.
Similarly, to prove that $\mathsf{PM}$ satisfies the equation $\mathsf{root}_p(x,r_p(x))$ for a prime $p$, it will be enough to show that so does $\mathsf{EPF}$. Then, consider $\A \in \mathsf{EPF}$ and $a \in A$. As $\A$ has a zero-totalized field reduct that is perfect, from \cref{Prop : graph} it follows that $\A \vDash \mathsf{root}_p(a, r_p(a))$. 

Next, let $\Sigma$ be the set of axioms in the statement. We will prove that every algebra $\A = \langle A; +, \cdot, -, 0, 1, (\,)^*, \{ r_p : \text{$p$ is prime}\}\rangle$ satisfying  $\Sigma$  belongs to $\mathsf{PM}$.
To this end, consider an algebra $\A$ as above satisfying $\Sigma$. In view of \cref{Thm : meadows axioms}, $\A$ has a meadow reduct $\A_\mathsf{m}$. Therefore, $\A_\mathsf{m} \in \mathsf{M} = \III\SSS\PPP(\mathsf{ZTF})$. Then, there exist $\{ \B_i : i \in I \} \subseteq \mathsf{ZTF}$ and an embedding $h \colon \A_\mathsf{m} \to \prod_{i \in I}\B_i$. By \cref{Cor : every field extend to a rooted one} each $\B_i$ extends to a perfect $\C_i \in \mathsf{ZTF}$. Since $\B_i \leq \C_i$ for each $i \in I$, we can view $h$ as an embedding $h \colon \A_\mathsf{m} \to \prod_{i \in I}\C_i$. Furthermore, as each $\C_i$ is perfect, it is the meadow reduct of some $\D_i \in \mathsf{EPF}$. Therefore, in order to conclude the proof, it only remains to show that $h \colon \A \to \prod_{i \in I}\D_i$ is also an embedding, for in this case we would have $\A \in \III\SSS\PPP(\{ \D_i : i \in I\}) \subseteq \III\SSS\PPP(\mathsf{EPF}) = \mathsf{PM}$, as desired.

First, observe that $h \colon  \A \to \prod_{i \in I}\D_i$ is a well-defined injective map that  preserves all the basic operations of meadows because $h \colon \A_\mathsf{m} \to \prod_{i \in I}\C_i$ is an embedding from the meadow reduct of $\A$ to the meadow reduct of $\prod_{i \in I}\D_i$. Therefore, to show that $h$ is an embedding of perfect meadows, it suffices to show that $h$ preserves the operations of the form $r_p$.
To this end, consider a prime $p$ and $a \in A$. We need to prove that $h(r_p^\A(a)) = r_p^{   \prod_{i \in I}\D_i}(h(a))$, that is, 
\begin{equation}\label{Eq : technical : ICM is V}
(h(r_p^\A(a)))(i) = r_p^{\D_i}((h(a))(i))\text{ for every }i \in I.
\end{equation}
In order to prove the above display, consider $j \in I$. First, recall that $\A$ satisfies the equation $\mathsf{root}_p(x,r_p(x))$ by assumption. Therefore, 
\[
(r_p^\A(a))^p =  (1-p^*p)a.
\]
Let $\pi_j \colon \prod_{i \in I}\C_i \to \C_j$ be the projection onto the $j$-th coordinate. 
As $\pi_j \circ h \colon \A_\mathsf{m} \to \C_j$ is a homomorphism of meadows, from the above display it follows that the following  holds in $\C_j$:
\[
((h(r^\A_p(a)))(j))^p = (1-p^*p)(h(a)(j)).
\]
Together with the definition of $\mathsf{root}_p(x, y)$, this yields
\[
\C_j \vDash \mathsf{root}_p((h(a))(j), (h(r_p^\A(a)))(j)). 
\]
Recall that $\C_i$ is perfect because it is the meadow reduct of $\D_i \in \mathsf{EPF}$. Consequently, we can apply \cref{Prop : graph} to the above display,  
obtaining that $(h(r_p^\A(a)))(i)$ is $r_p((h(a))(i))$ in $\C_i$. As $\C_i$ is the meadow reduct of $\D_i$, we have $(h(r_p^\A(a)))(i) = r_p^{\D_i}((h(a))(i))$, whence (\ref{Eq : technical : ICM is V}) holds.
\end{proof}

\subsection{Subdirect irreducibility}

Let $\K$ be a quasivariety. A congruence $\theta$ of an algebra $\A \in \K$ is said to be a $\K$-\emph{congruence} of $\A$ when $\A / \theta \in \K$. When ordered under inclusion, the set of $\K$-congruences of $\A$ forms an algebraic lattice $\Con_\K(\A)$, whose minimum is the identity relation $\mathsf{id}_A$ on $A$ (see, e.g.,  \cite[Prop.~1.4.7 \& Cor.~1.4.11]{Go98a}). We recall that an element $a$ of a bounded lattice $\A$ is  \emph{meet irreducible} when it is not the maximum of $\A$ and for all $b, c \in A$ such that $ a= b \land c$ either $a = b$ or $a = c$ holds. Then, we say that a member $\A$ of $\K$ is \emph{relatively finitely subdirectly irreducible} when  $\mathsf{id}_A$ is  meet irreducible in $\Con_\K(\A)$ (see, e.g., \cite[p.\ 659]{MR2428150}). 
The class of relatively finitely subdirectly irreducible members of $\K$ will be denoted by $\K_\textsc{rfsi}$. When $\K$ is a variety, the class $\K_\textsc{rfsi}$ will be denoted by $\K_\textsc{fsi}$. The importance of $\K_\textsc{rfsi}$ derives from Birkhoff and Maltsev's subdirect decomposition theorem, which ensures that $\K = \III\SSS\PPP(\K_\textsc{rfsi})$ for every quasivariety $\K$ (see, e.g., \cite[Thm.\ 3.1.1]{Go98a}).

\begin{Theorem}[\protect{\cite[Lem.~1.5]{CD90}}]\label{Thm : RFSI members of Q(K)}
Let $\K$ be a class of similar algebras. Then $\QQQ(\K)_\textsc{rfsi} \subseteq \III\SSS\PPU(\K)$.
\end{Theorem}

Next, we recall that an algebra $\A$ is \emph{simple} when it has exactly two congruences, namely, $\mathsf{id}_A$ and $A \times A$. For instance, every field is simple and, therefore, so is every field expansion. We also recall that $\mathsf{PM}$ is a variety by  \cref{Thm : ICM axioms}. This explains why, in the next result, we write $\mathsf{PM}_\textsc{fsi}$ instead of $\mathsf{PM}_\textsc{rfsi}$.

\begin{Proposition}\label{Prop : generation for RCR and ICM}
The following equalities hold:
\[
\mathsf{RCR} =  \QQQ(\mathsf{PF}), \quad \mathsf{PM} = \QQQ(\mathsf{EPF}), \quad \mathsf{RCR}_\textsc{rfsi} = \mathsf{ID} = \SSS(\mathsf{F}), \quad \mathsf{PM}_\textsc{fsi} = \mathsf{EPF}.
\]
\end{Proposition}

\begin{proof}
We begin by proving $\mathsf{RCR} = \QQQ(\mathsf{PF})$. Observe that
\[
\mathsf{RCR} = \III\SSS\PPP(\mathsf{F}) \subseteq \III\SSS\PPP\SSS(\mathsf{PF}) = \III\SSS\PPP(\mathsf{PF}) \subseteq \III\SSS\PPP(\mathsf{F}) = \mathsf{RCR},
\]
where the first and the last steps hold by definition,
the second by Corollary \ref{Cor : every field extend to a rooted one}, the third because $\III\SSS\PPP(\K)$ is closed under $\SSS$ for every class of algebras $\K$ (see, e.g., \cite[Thm.\ V.2.20]{BuSa00}), and the fourth because $\mathsf{PF}\subseteq \mathsf{F}$. Therefore, $\mathsf{RCR} = \III\SSS\PPP(\mathsf{PF})$. Since $\mathsf{PF}$ is closed under $\PPU$ by \cref{Prop : ICF closed under S and Pu},
this yields $\mathsf{RCR} = \III\SSS\PPP\PPU(\mathsf{PF})$. By \cref{Thm : Q(K) and V(K)} we conclude that $\mathsf{RCR} = \QQQ(\mathsf{PF})$, as desired.

Next, observe that
\[
\mathsf{PM} = \III\SSS\PPP(\mathsf{EPF}) = \III\SSS\PPP\PPU(\mathsf{EPF}) = \QQQ(\mathsf{EPF}),
\]
where the first equality holds by definition, the second because $\mathsf{EPF}$ is closed under $\PPU$ by \cref{Prop : ICF closed under S and Pu}, and the third by \cref{Thm : Q(K) and V(K)}. 

Lastly, recall that $\mathsf{ID} = \SSS(\mathsf{F})$ by definition.
As $\mathsf{RCR}_\textsc{rfsi} = \mathsf{ID}$ (see, e.g., \cite[p.~410]{CampRaf}),
it only remains to show that $\mathsf{PM}_\textsc{fsi} = \mathsf{EPF}$. From $\mathsf{PM} = \QQQ(\mathsf{EPF})$ and \cref{Thm : RFSI members of Q(K)} it follows that $\mathsf{PM}_\textsc{fsi} \subseteq \III\SSS\PPU(\mathsf{EPF}) = \mathsf{EPF}$, where the last equality holds by \cref{Prop : ICF closed under S and Pu}. 
 To prove the reverse inclusion, consider $\A \in \mathsf{EPF} \subseteq \mathsf{PM}$. As $\A$ has a field reduct, it is simple. Consequently, $\mathsf{id}_A$ is meet irreducible in $\Con_{\mathsf{PM}}(\A)$ and $\A \in \mathsf{PM}_\textsc{fsi}$.
\end{proof}

\section{Primitive positive expansions}\label{Sec : Beth companions}

As we mentioned, in order to complete a quasivariety so that its monomorphisms become regular, it is sensible to expand it with enough implicit operations. In the context of quasivarieties, implicit operations admit a  description in terms of \emph{primitive positive formulas} (for short, \emph{pp formulas}), that is, formulas of the form $\exists x_1, \dots, x_n \varphi$, where $\varphi$ is a conjunction of equations. More precisely,  if $f$ is an implicit operation of a quasivariety $\K$, there exist implicit operations $f_1, \dots, f_n$ of $\K$ definable by pp formulas such that $f^\A = f_1^\A \cup \dots \cup f_n^\A$ for every $\A \in \K$ (see \cite[Cor.\ 3.10]{CKMIMPv2}).
Consequently, the pp definable implicit operations of $\K$ form the building blocks of all implicit operations of $\K$ and, therefore, we restrict our attention to them. 

In general, the implicit operations of a quasivariety  need not be total. This motivates  the following concept (see \cite[Sec.\ 8]{CKMIMPv2}).

\begin{Definition}
Let $\K$ be a quasivariety. An implicit operation $f$ of $\K$ is \emph{extendable} when for all $\A \in \K$ and  $\langle a_1, \dots, a_n \rangle \in \mathsf{dom}(f^\A)$ there exists an algebra $\B \in \K$ with $\A \leq \B$ such that $\langle a_1, \dots, a_n \rangle \in \mathsf{dom}(f^\B)$. The class of extendable implicit operations of $\K$ will be denoted by $\ext(\K)$, and that of pp definable extendable implicit operations of $\K$ by $\extpp(\K)$. 
\end{Definition}
The term ``extendable'' in the above definition derives from the fact that every member of a quasivariety $\K$ can be 
upgraded to one in which all the extendable implicit operations are total in the sense that for every $\A \in \K$ there exists $\B \in \K$ with $\A \leq \B$ such that $f^\B$ is total and extends $f^\A$ for each $f \in \ext(\K)$ (see   \cite[Prop.\ 8.1 \& Thm.\ 8.4]{CKMIMPv2}).
This property of an implicit operation becomes crucial if we want to add it to $\K$.

We will show that “taking weak inverses” and “taking weak prime roots” can be viewed as extendable implicit operations of $\mathsf{RCR}$. To this end, we consider the following pp formulas in the language of $\mathsf{RCR}$ for every prime $p$:
\[
\mathsf{inv}(x, y) = (x^2y \thickapprox x) \sqcap (x y^2 \thickapprox y) \, \, \text{ and } \, \, \exists\mathsf{root}_p(x, y) = \exists z ( \mathsf{inv}(p, z) \sqcap (y^p \thickapprox (1-zp)x)).
\]

\begin{Theorem}\label{Thm : root and inv are extendable}
The following conditions hold for every prime $p$ and $\A \in \mathsf{PF}$:
\benroman
\item the formula $\mathsf{inv}(x, y)$ defines a unary $g \in \extpp(\mathsf{RCR})$ such that $g^\A$ is total and 
\[
g^\A(a) = a^*\text{ for every }a \in A;
\]
\item the formula $\exists \mathsf{root}_p(x, y)$ defines a unary $f_p \in \extpp(\mathsf{RCR})$ such that $f_p^\A$ is total and 
\[
f_p^\A(a) = r_p(a) \text{ for every }a \in A.
\]
\eroman 
\end{Theorem}

The proof of \cref{Thm : root and inv are extendable} is based on the next observation (see 
\cite[Cor.\ 3.11 \& Prop.\ 8.11(ii)]{CKMIMPv2}).

\begin{Proposition}\label{Prop : extendable trick} Let $\K$ be an elementary class of similar algebras and $\varphi(x_1, \dots, x_n, y)$ a pp formula satisfying the following conditions:
\benroman
\item\label{item : extendable trick : 1} $\K \vDash (\varphi(x_1, \dots, x_n, y) \sqcap \varphi(x_1, \dots, x_n, z)) \to y \thickapprox z$;
\item\label{item : extendable trick : 2} for all $\A \in \K$ and $a_1, \dots, a_n \in A$ there exist $\B \in \QQQ(\K)$ and $b \in B$ such that $\A \leq \B$ and $\B \vDash \varphi(a_1, \dots, a_n, b)$. 
\eroman
Then $\varphi$ defines an $n$-ary member of $\extpp(\QQQ(\K))$.
\end{Proposition}

We are now ready to prove \cref{Thm : root and inv are extendable}.

\begin{proof}
We begin with the following observation.

\begin{Claim}\label{Claim : extendability}
For all $\A \in \mathsf{PF}$ and $a, b \in A$,
\[
(\A \vDash \mathsf{inv}(a, b) \iff b = a^*) \, \,  \text{ and } \, \,   (\A \vDash \exists \mathsf{root}_p(a, b) \iff b = r_p(a)) .
\]
\end{Claim}

\begin{proof}[Proof of the Claim]
To prove the left hand side of the above display, observe that $b = a^*$ immediately implies $\A \vDash \mathsf{inv}(a, b)$. Conversely, suppose that $\A \vDash \mathsf{inv}(a, b)$, that is, $a^2 b = a$ and $ab^2 = b$. First, suppose that $a = 0$. Then $a = a^* = 0$. 
Consequently, from $ab^2 = b$ it follows that $a^* = 0 = a b^2 = b$. Next, we consider the case where $a \ne 0$ and, therefore, $a^* = a^{-1}$. Thus, from $a^2 b = a$ it follows that $b = a^{-2}a^2 b = a^{-2}a = a^{-1} = a^*$,
establishing the left hand side of the above display.

Next we prove the right hand side. Since $\A$ is a field, we can expand it to a zero totalized field $\A^+$.
We will prove that 
\begin{align*}
    \A \vDash \exists \mathsf{root}_p(a, b) &\iff \A \vDash \mathsf{inv}(p1, c) \sqcap (b^p \thickapprox (1-cp)a) \text{ for some }c \in A\\
  & \iff  b^p = (1-p^*p)a\\
  & \iff \A^+ \vDash \mathsf{root}_p(a, b)\\
  &\iff b = r_p(a).
\end{align*}
The above equivalences are justified as follows: the first holds by the definition of $\exists \mathsf{root}_p$, the second by the left hand side of the display in the statement of the claim, the third by the definition of $\mathsf{root}_p(x, y)$, and the fourth by \cref{Prop : graph} and the assumption that $\A$ is perfect.
\end{proof}

Recall from \cref{Prop : ICF closed under S and Pu} that $\mathsf{PF}$ is an elementary class. Therefore, we can apply \cref{Claim : extendability}, obtaining that condition (\ref{item : extendable trick : 1}) of \cref{Prop : extendable trick} holds in $\mathsf{PF}$
both for $\mathsf{inv}(x, y)$ and $\exists \mathsf{root}_p(x, y)$. 
The same is true for condition (\ref{item : extendable trick : 2}) of \cref{Prop : extendable trick} by \cref{Claim : extendability} and the fact that for every $\A \in \mathsf{PF}$ and $a \in A$ the elements $a^*$ and $r_p(a)$ exist in $\A$. Therefore, from \cref{Prop : extendable trick} it follows that there exist unary $g,f_p \in \extpp(\QQQ(\mathsf{PF}))$ defined by $\mathsf{inv}(x, y)$ and $\exists \mathsf{root}_p(x, y)$, respectively.
As $\mathsf{RCR} = \QQQ(\mathsf{PF})$
by Proposition \ref{Prop : generation for RCR and ICM}, we obtain $g, f_p \in \extpp(\mathsf{RCR})$. Lastly, let $\A \in \mathsf{PF}$ and recall that $a^*$ and $r_p(a)$ exist in $\A$ for every $a \in A$. Together with \cref{Claim : extendability}, this implies that $g^\A$ and $f_p^\A$ are total and such that $g^\A(a) = a^*$ and $f_p^\A(a) = r_p(a)$ for every $a \in A$.
\end{proof}

In order to add a family of implicit operations $\mathcal{F} \subseteq \extpp(\K)$ to a quasivariety $\K$, we proceed as follows. Let $\L$ be the language of $\K$ and $\L_\F$ the language obtained by adding to $\L$ a new $n$-ary function symbol $g_f$ for each $n$-ary $f \in \F$. Then, we expand every member $\A$ of $\K$ in which $\{ f^\A : f \in \F \}$ is a family of total functions to an algebra $\A[\L_\F]$ in the language $\L_\F$ by interpreting  $g_f$ as $f^\A$ for each $f \in \F$. In addition, for every $\mathsf{N} \subseteq \K$  let
\[
\mathsf{N}[\L_\F] = \{ \A[\L_\F] : \A \in \mathsf{N} \text{ and }\{ f^\A : f \in \F \} \text{ is a family of total functions} \}.
\]
 The next definition captures the idea of  ``adding the implicit operations in $\mathcal{F}$ to $\K$''.
 
\begin{Definition}
Let $\K$ be a quasivariety.
\benroman
\item Given $\mathcal{F} \subseteq \extpp(\K)$, the \emph{pp expansion} of $\K$ induced by $\F$ is $\SSS(\K[\L_\F])$. 
\item A pp expansion of $\K$ is said to be \emph{simple} when it is of the form $\K[\L_\F]$ for some $\F \subseteq \extpp(\K)$.
\eroman
\end{Definition}

Notably, every pp expansion of a quasivariety is also a quasivariety  (see \cite[Thm.\ 10.3(ii)]{CKMIMPv2}).  Simple pp expansions admit a transparent categorical description, as we proceed to recall. 
An isomorphism-closed
full subcategory $\mathsf{C}$ of a category $\mathsf{D}$ is said \emph{mono-reflective} when the inclusion functor $i \colon \mathsf{C} \to \mathsf{D}$ has a left adjoint and the unit of resulting adjunction is componentwise a monomorphism (see, e.g., \cite[Def.~16.1]{AHS06}). 

\begin{Theorem}[\protect{\cite[Thm.~3.1]{BethCat26}}]\label{Thm : main pp}
Let $\K$ and $\M$ be a pair of quasivarieties. Then  $\M$ is a simple pp expansion of $\K$ if and only if the forgetful functor $U \colon \M \to \K$ is well defined and induces an isomorphism from $\M$ to a mono-reflective subcategory of $\K$.
\end{Theorem}

For the present purpose, the pp expansions of interest are those obtained by adding enough implicit operations so that every monomorphism becomes regular. 
\begin{Definition}
Let $\K$ be a quasivariety.
\benroman
\item A pp expansion of  $\K$ is said to be a \emph{Beth companion} of $\K$ when its monomorphisms are regular.
\item A Beth companion of $\K$ is said to be \emph{simple} when it is simple as a pp expansion of $\K$.
\eroman
\end{Definition}

While a quasivariety need not have a Beth companions, up to term equivalence it may possess only one  (see \cite[Thm.\ 11.7]{CKMIMPv2}).
For this reason, we talk about \emph{the} Beth companion of $\K$ (when it exists).
The aim of this paper is to establish that $\mathsf{PM}$ is the Beth companion of $\mathsf{RCR}$ and that, moreover, this Beth companion is simple (see \cref{Thm : main}). As a first step, in this section we will prove  the following result. Let
\[
\F = \{ g \} \cup \{ f_p : \text{$p$ is prime} \},
\]
where $g$ and $f_p$ are the members of $\extpp(\mathsf{RCR})$ given by \cref{Thm : root and inv are extendable}.

\begin{Theorem}\label{Thm : ICM is pp expansion of RCR}
The pp expansion of $\mathsf{RCR}$ induced by $\mathcal{F}$ is simple and coincides with $\mathsf{PM} = \mathsf{RCR}[\mathscr{L}_{\mathcal{F}}]$.
\end{Theorem}

 To this end, we will employ the next observation  (see \cite[Thm.\ 10.5]{CKMIMPv2}).

\begin{Proposition}\label{Prop : trick pp expansions}
Let $\K$ be a class of similar algebras and $\F \subseteq \extpp(\QQQ(\K))$. Assume that $f^\A$ is total for every $f \in \F$ and $\A \in \K$. Then $\QQQ(\K[\L_\F])$ is the pp expansion of $\QQQ(\K)$ induced by $\mathcal{F}$.
\end{Proposition}

We are now ready to prove \cref{Thm : ICM is pp expansion of RCR}.

\begin{proof}
We begin by proving that the pp expansion of $\mathsf{RCR}$ induced by $\mathcal{F}$ coincides with $\mathsf{PM}$. To this end, recall from \cref{Thm : root and inv are extendable} that $g^\A = (\,)^*$ and $f_p^\A = r_p$  are total for all $\A \in \mathsf{PF}$  and prime $p$.
Moreover, $\mathsf{RCR} = \QQQ(\mathsf{PF})$ by \cref{Prop : generation for RCR and ICM}. Therefore, we can apply \cref{Prop : trick pp expansions} to $\K = \mathsf{PF}$ and $\mathcal{F}$, obtaining  that  $\QQQ(\mathsf{PF}[\L_\F])$ is the pp expansion of $\mathsf{RCR}$ induced by $\mathcal{F}$. 
As both $\mathsf{PF}[\L_\F]$ and $\mathsf{EPF}$ consist of the members of $\mathsf{PF}$ expanded with the operations $(\,)^*$ and $\{r_p : p \text{ prime}\}$, we conclude that $\mathsf{PF}[\L_\F] = \mathsf{EPF}$. Thus, $\QQQ(\mathsf{EPF})$ is the pp expansion of $\mathsf{RCR}$ induced by $\mathcal{F}$. Moreover, it coincides with $\mathsf{PM}$ by \cref{Prop : generation for RCR and ICM}. Therefore, $\mathsf{PM}$ is the pp expansion of $\mathsf{RCR}$ induced by $\mathcal{F}$, as desired.

In order to prove that this pp expansion is simple, it suffices to show that $\mathsf{PM} = \mathsf{RCR}[\L_\F]$. As $\mathsf{PM} = \SSS(\mathsf{RCR}[\mathscr{L}_{\mathcal{F}}])$,  it only remains to show that $\mathsf{RCR}[\mathscr{L}_{\mathcal{F}}]$ is closed under subalgebras. To this end, consider
$\A \leq \B \in \mathsf{RCR}[\mathscr{L}_{\mathcal{F}}]$. We have to show that $f^{\A^-}$ is total for every $f \in \mathcal{F}$. To this end, will use repeatedly the fact that the maps $g^{\B^-}$ and $f_p^{\B^-}$ are total and defined for every $a \in B$ as
\[
g^{\B^-}(a) = a^{*\B} \, \, \text{ and }\, \, f_p^{\B^-}(a) = r_p^{\B}(a),
\]
which holds because $\B \in \mathsf{RCR}[\mathscr{L}_{\mathcal{F}}]$.

First, let $f = g$ and $a \in A$. Since $\B \in  \mathsf{RCR}[\mathscr{L}_{\mathcal{F}}]$ and $\mathsf{inv}(x, y)$ defines $g$ by \cref{Thm : root and inv are extendable}, we obtain $\B \vDash \mathsf{inv}(a, a^{*\B})$. As $a \in A$ and $\A \leq \B$, we have $a^{*\B} \in A$. 
Since $\mathsf{inv}(a, a^{*\B})$ is a conjunction of equations without free variables, $a,a^{*\B} \in A$, and $\A \leq \B$, we obtain $\A \vDash \mathsf{inv}(a,a^{\B*})$.
Thus,  
$a \in \mathsf{dom}(g^{\A^-})$ because $\mathsf{inv}(x, y)$ defines $g$.

       Next, we consider the case where $f = f_p$ for some prime $p$. Let $a \in A$. Since $\B \in  \mathsf{RCR}[\mathscr{L}_{\mathcal{F}}]$ and $\exists\mathsf{root}_p(x, y)$ defines $f_p$ by \cref{Thm : root and inv are extendable}, we obtain $\B \vDash \exists\mathsf{root}_p(a, r_p^{\B}(a))$. By the definition of $\exists\mathsf{root}_p(x, y)$ there exists $c \in B$ such that 
       \[
       \B \vDash \mathsf{inv}(p, c) \sqcap (r_p^{\B}(a)^p \thickapprox (1-cp) a).
       \]
       As $\mathsf{inv}(x, y)$ defines $g$, the above display amounts to
       \[
       \B \vDash \mathsf{inv}(p,p^{*\B}) \sqcap (r_p^{\B}(a)^p \thickapprox (1-p^{*\B}p)a).
       \]
Observe that $a, p \in A$. Together with $\A \leq \B$, this ensures $r_p^{\B}(a), p^{*\B} \in A$. As equations are preserved by subalgebras,
the above display implies
       \[
       \A \vDash \mathsf{inv}(p,p^{*\B}) \sqcap (r_p^{\B}(a)^p \thickapprox (1-p^{*\B}p)a ).
       \]
  Hence, we conclude that $a \in \mathsf{dom}(f_p^{\A^-})$ because $\exists\mathsf{root}_p(x, y)$ defines $f_p$.
\end{proof}

\begin{Definition}
We denote by $\A^-$ the ring reduct of  a perfect meadow.
\end{Definition}

From \cite[Prop.\ 10.2]{CKMIMPv2} and \cref{Thm : ICM is pp expansion of RCR} we deduce the following.

\begin{Corollary}\label{Cor : subreducts}
The following conditions hold:
\benroman
\item\label{item : subreducts : 1} for every $\A \in \mathsf{PM}$ we have $\A^- \in \mathsf{RCR}$;
\item\label{item : subreducts : 2} for every $\A \in \mathsf{RCR}$ there exists $\B \in \mathsf{PM}$ such that $\A \leq \B^-$. 
\eroman
\end{Corollary}

We will also make use of the following observation.

\begin{Corollary} \label{Cor : RCR sub are ICM sub}
Let $\A, \B \in \mathsf{PM}$. Every ring homomorphism $h \colon \A^- \to \B^-$ is also a homomorphism $h \colon \A \to \B$ of perfect meadows.
\end{Corollary}

\begin{proof}
 Recall from \cref{Thm : ICM is pp expansion of RCR} that $\mathsf{PM} = \mathsf{RCR}[\mathscr{L}_{\mathcal{F}}]$. By  \cite[Prop.\ 9.5]{CKMIMPv2} the latter yields the desired conclusion.
\end{proof}

\section{Amalgamation}

\begin{Definition}
Given a class $\K$ of similar algebras, we say that
\benroman
\item a tuple
$\langle \A, \B, \C, h_1, h_2 \rangle$ is a \emph{span}
in $\K$ when $h_1 \colon \A \to \B$ and $h_2 \colon \A \to \C$ is a pair of embeddings with $\A, \B, \C \in \K$;
\item a span
$\langle \A, \B, \C, h_1, h_2 \rangle$ in $\K$ 
has an \emph{amalgam}
in $\K$ when there exists a pair of embeddings $g_1 \colon \B \to \D$ and $g_2 \colon \C \to \D$ with $\D \in \K$ such that $g_1 \circ h_1 = g_2 \circ h_2$;
\item a member $\A$ of $\K$ is an \emph{amalgamation base} for $\K$ when every span
in $\K$ of the form $\langle \A, \B, \C, h_1, h_2 \rangle$ 
has an amalgam
in $\K$;
\item $\K$ has the \emph{amalgamation property} when every 
span
in $\K$ 
has an amalgam
in $\K$.
\eroman
\end{Definition}

The aim of this section is to establish the following.

\begin{Theorem}\label{Thm : ICM has AP}
    $\mathsf{PM}$ has the amalgamation property.
\end{Theorem}

This result contrasts with the case of $\mathsf{RCR}$, a class that lacks the amalgamation property (see, e.g., \cite[p.\ 426]{CorAPR}, where reduced commutative rings are called ``commutative semiprime rings with identity'').
The next concept will be instrumental for proving \cref{Thm : ICM has AP} (see \cite[Def.~4]{vNeu36}).

\begin{Definition}
A commutative ring $\A$ is said to be \emph{(von Neumann) regular} when for every $a \in A$ there exists $b \in A$ such that  $a=a^2b$.
\end{Definition}
It is immediate that every regular commutative ring is reduced. Moreover, we have the following.

\begin{Proposition}\label{Prop : ICM reducts are regular}
    Let $\A \in \mathsf{PM}$. Then $\A^-$ is a regular commutative ring.
\end{Proposition}

\begin{proof}
    Clearly, $\A^-$ is a commutative ring. To prove that it is regular, consider $a \in A$ and let $b = a^*$. From \cref{Thm : ICM axioms} it follows that $a =  a^2a^*  = a^2 b$. 
\end{proof}

For the present purpose, the importance of regular rings derives from the following fact (see, e.g., \cite[Thm.~1.6]{CorAPR}). 

\begin{Theorem} \label{Thm : regular rings have AP}
    The class of regular commutative rings has the amalgamation property. 
\end{Theorem}

We are now ready to prove \cref{Thm : ICM has AP}.

\begin{proof}
Let $\langle \A, \B, \C, h_1, h_2 \rangle$ be 
a span
in $\mathsf{PM}$. We need to show that it 
has an amalgam
in $\mathsf{PM}$. First, observe that $\langle \A^-, \B^-, \C^-, h_1, h_2 \rangle$ is 
a span
in the class of regular commutative rings by  \cref{Prop : ICM reducts are regular}. By \cref{Thm : regular rings have AP} there exist a regular commutative ring $\D$ and a pair of ring embeddings $g_1 \colon \B^- \to \D$ and $g_2 \colon \C^- \to \D$ such that $g_1 \circ h_1 = g_2 \circ h_2$. Next, recall that every regular commutative ring is reduced, whence $\B^-, \C^-, \D \in \mathsf{RCR}$.
 In view of Corollary~\ref{Cor : subreducts}\eqref{item : subreducts : 2}, we may assume that $\D = \boldsymbol{E}^-$ for some $\boldsymbol{E} \in \mathsf{PM}$, and view $g_1$ and $g_2$ as  
 ring embeddings $g_1 \colon \B^- \to \boldsymbol{E}^-$ and $g_2 \colon \C^- \to \boldsymbol{E}^-$ with $\B, \C, \boldsymbol{E} \in \mathsf{PM}$. Hence, we can apply \cref{Cor : RCR sub are ICM sub}, obtaining that $g_1$ and $g_2$ can also be viewed as homomorphisms $g_1 \colon \B \to \boldsymbol{E}$ and $g_2 \colon \C \to \boldsymbol{E}$ of perfect meadows. Since $g_1 \circ h_1 = g_2 \circ h_2$, we conclude that   $\langle \A, \B, \C, h_1,h_2 \rangle$ 
 has an amalgam 
 in $\mathsf{PM}$.
\end{proof}

Recall that every regular commutative ring is reduced.

\begin{Proposition} \label{Prop : regular rings are Amalagamtion bases for RCR}
    Every regular commutative ring is an
    amalgamation base for $\mathsf{RCR}$.
\end{Proposition}

\begin{proof}
    Let $\langle \A, \B, \C, h_1, h_2 \rangle$ be 
    a span
    in $\mathsf{RCR}$ with $\A$ a regular commutative ring. We need to show that it 
     has an amalgam
    in $\mathsf{RCR}$.
    From \cref{Cor : subreducts}\eqref{item : subreducts : 2} and \cref{Prop : ICM reducts are regular} it follows that there exists a pair of regular commutative rings $\B'$ and $\C'$ such that $\B \leq \B'$ and $\C \leq \C'$. Therefore, $\langle \A, \B', \C', h_1, h_2 \rangle$ is 
     a span 
    in the class of regular commutative rings. By \cref{Thm : regular rings have AP} there exists a pair of embeddings $g_1 \colon \B' \to \D$ and $g_2 \colon \C' \to \D$ with $\D$ a regular commutative ring such that $g_1 \circ h_1 = g_2 \circ h_2$. Clearly, the restrictions $g_1^* \colon \B \to \D$ and $g_2^* \colon \C \to \D$ are also embeddings such that  $g_1^* \circ h_1 = g_2^* \circ h_2$. As $\D$ is reduced (because it is regular), we are done.
\end{proof}

\section{The main result}

The aim of this section is to prove our main result, which takes the following form.

\begin{Theorem}\label{Thm : main}
    $\mathsf{PM}$ is the Beth companion of $\mathsf{RCR}$ and, as such, it is simple.
\end{Theorem}

As a consequence of \cref{Thm : main}, we obtain a categorical description of perfect meadows.

\begin{Corollary}\label{Cor : categorical description}
    Up to term equivalence, $\mathsf{PM}$ is    the unique quasivariety $\K$ in which monomorphisms are regular for which the forgetful functor $U \colon \K \to \mathsf{RCR}$ is well defined and induces an isomorphism from $\K$ to a mono-reflective subcategory of $\mathsf{RCR}$.
\end{Corollary}

\begin{proof}
    As Beth companions are unique up to term equivalence (see \cite[Thm.\ 11.7]{CKMIMPv2}), \cref{Thm : main} implies that $\mathsf{PM}$ is (up to term equivalence) the only simple pp expansion of $\mathsf{RCR}$ in which monomorphisms are regular. Therefore, the desired results follow from 
    \cref{Thm : main pp}.
\end{proof}

The remainder of this section is devoted to proving \cref{Thm : main}. Recall from \cref{Thm : ICM is pp expansion of RCR} that $\mathsf{PM}$ is a simple pp expansion of $\mathsf{RCR}$.\ Therefore, the proof of \cref{Thm : main} reduces to verifying the following.

\begin{Theorem} \label{Thm : ICM has SES}
 Monomorphisms are regular in $\mathsf{PM}$. 
\end{Theorem}

 The next concepts will be instrumental in proving \cref{Thm : ICM has SES}. We begin by reviewing the notion of a dominion due to Isbell (see \cite{Isb65}). 

\begin{Definition}
Given a class of similar algebras $\K$ and $\A \leq \B \in \K$, the \emph{dominion} of $\A$ in $\B$ relative to $\mathsf{K}$
is the set
\begin{align*}
\mathsf{d}_\mathsf{K}(\A, \B) = \{ & b \in B : g(b) = h(b) \text{ for every pair of homomorphisms} \\
&g, h \colon \B \to \C  \text{ with }\C \in \mathsf{K} \text{ such that }g{\upharpoonright}_A = h{\upharpoonright}_A\}.
\end{align*}
\end{Definition}

We will also rely on the following concept.

\begin{Definition}
A variety $\mathsf{K}$ is  \emph{congruence permutable} when $\theta \circ \phi = \phi \circ \theta$ for all $\A \in \K$ and $\theta, \phi \in \Con (\A)$, where $\circ$ is the operation of composition of relations. 
\end{Definition}

\begin{Remark}\label{Rmk : dominions vs SES}
We recall that every variety with a group reduct is 
congruence permutable (see, e.g., \cite[p.\ 79, Example.\ 1]{BuSa00}).

In particular, $\mathsf{PM}$ is congruence permutable.
\qed
\end{Remark}

The proof of \cref{Thm : ICM has SES} is facilitated 
by the following result (see  \cite[Cor.\ 7.16]{CKMIMPv2}). 

\begin{Theorem}\label{Thm : our CP theorem}
Let $\K$ be a congruence permutable variety with the amalgamation property. Then monomorphisms are regular in $\K$ if and only if $\mathsf{d}_\K(\A, \B) = A$ for all $\A \leq \B \in \mathsf{K}_\textsc{fsi}$. 
\end{Theorem}

As $\mathsf{PM}$ is a congruence permutable variety with the amalgamation property (see \cref{Rmk : dominions vs SES} and \cref{Thm : ICM has AP}), it falls within the scope of \cref{Thm : our CP theorem}. Keeping in mind that  $\mathsf{PM}_\textsc{fsi} = \mathsf{EPF}$ (see \cref{Prop : generation for RCR and ICM}), this implies that the proof of  \cref{Thm : ICM has SES} reduces to verifying the following.

\begin{Proposition}\label{Prop : main}
    For all $\A \leq \B \in \mathsf{EPF}$ we have $\mathsf{d}_\mathsf{PM}(\A, \B) = A$.
\end{Proposition}

To this end, we will utilize the next observation (see \cite[Prop.\ 4.11]{CKMIMPv2}).

\begin{Proposition}
\label{Prop : doms computable in subalgebras amalgamation base}
   Let $\mathsf{K}$ be a class of similar algebras closed under finite direct products and $\A \leq \B \leq \C$ with $\B, \C \in \K$.  If $\B$ is an amalgamation base for $\mathsf{K}$, then 
   \[
   \mathsf{d}_{\mathsf{K}}(\A,\B) = \mathsf{d}_{\mathsf{K}}(\A,\C) \cap B.
   \]
\end{Proposition}

We are now ready to prove \cref{Prop : main}. This will establish Theorems \ref{Thm : main} and \ref{Thm : ICM has SES} as well.

\begin{proof} 
Consider $\A \leq \B \in \mathsf{EPF}$. As the inclusion $A \subseteq \mathsf{d}_\mathsf{PM}(\A,  \B)$ is an immediate consequence of the definition of a dominion, we detail only the proof of the reverse inclusion. To this end, we reason by contraposition. Consider $b \in B - A$. We will show that $b \notin \mathsf{d}_\mathsf{PM}(\A,  \B)$. 

As $\A \leq \B \in \mathsf{EPF}$ by assumption and $\mathsf{EPF}$ is closed under subalgebras by \cref{Prop : ICF closed under S and Pu}, we obtain $\A, \B \in \mathsf{EPF}$. Together with $\A \leq \B$, this yields that 
\begin{equation}\label{Eq : A leq B is field extension in WRF}
    \A^- \leq \B^-\text{ is a field extension with }\A^-, \B^- \in \mathsf{PF}.
\end{equation}
 We begin with the following observation.

\begin{Claim}\label{Claim : very last argument of the main result}
We have $b \notin \mathsf{d}_{\mathsf{RCR}}(\A^-,\B^-)$.
\end{Claim}

\begin{proof}[Proof of the Claim]
Observe that  $\A^-(b)$ is regular because it is a field. Consequently, it is an amalgamation base for $\mathsf{RCR}$ by \cref{Prop : regular rings are Amalagamtion bases for RCR}. 
This allows us to apply \cref{Prop : doms computable in subalgebras amalgamation base}, obtaining 
$\mathsf{d}_{\mathsf{RCR}}(\A^-, \A^-(b)) = \mathsf{d}_{\mathsf{RCR}}(\A^-, \B^-) \cap A^-(b)$. 
Therefore, to establish the claim, 
it suffices to show that $b \notin \mathsf{d}_{\mathsf{RCR}}(\A^-, \A^-(b))$.

To this end, we consider the field extension $\A^- \leq \A^-(b)$ (see \eqref{Eq : A leq B is field extension in WRF} if necessary).   We  distinguish two cases depending on whether $b$ is algebraic or transcendental over $\A^-$. First, suppose that $b$ is algebraic over $\A^-$. 
Let $\mu_b$ be the minimal polynomial of $b$ over $\A^-$.
Since $b \notin A$ by assumption, we have $\mathsf{deg}(\mu_b) \geq 2$. 
  Recall that $\A^- \in \mathsf{PF}$ by \eqref{Eq : A leq B is field extension in WRF}. 
The
  field extension $\A^- \leq \A^-(b)$ is algebraic because $b$ is algebraic over $\A^-(b)$ by assumption (see, e.g., \cite[F5 p.\ 18]{FLField}).
    As $\A^-$ is perfect, this yields that $\A^- \leq \A^-(b)$ is separable. Since $\mathsf{deg}(\mu_b) \geq 2$, there exist a pair of distinct roots $c$ and $d$ of $\mu_b$ in $\mathsf{acl}(\A^-)$.
    By \cref{Prop : special field ext homs}\eqref{item : field ext hom alg} 
    there also exists a pair of embeddings $g, h \colon \A^-(b) \to \A^-(c, d)$ such that 
      \[
    g {\upharpoonright}_A = h {\upharpoonright}_A \, \, \text{ and } \, \, g(b) = c \neq d = h(b).
    \]
As $\A^-(c, d)$ is a field, the above display yields $b \notin \mathsf{d}_\mathsf{RCR}(\A^-, \A^-(b))$, as desired.

Next, we consider the case where $b$ is transcendental over $\A^-$. By \cref{Prop : special field ext homs}\eqref{item : field ext hom trans}
 there exists an isomorphism $g_1 \colon  \A^-(b)  \to \A^-(x)$ such that $g_1(b) = x$ and $g_1(a) = a$ for every $a \in A$. 
    By the same token there exists an isomorphism $h_1 \colon \A^-(b) \to \A^-(y)$ such that $h_1(b) = y$ and $h_1(a) = a$ for every $a \in A$.
 Observe that $\A^-[x], \A[y]^- \leq \A^-[x, y] \leq \A^-(x, y)$. Therefore, by \cref{Prop : embedding fields of fractions} there exists a pair of embeddings $g_2 \colon \A^-(x) \to \A^-(x, y)$ and $h_2 \colon \A^-(y) \to \A^-(x, y)$ such that $g_2(x) = x$, $h_2(y) = y$, and $g_2(a) = h_2(a) = a$ for every $a \in A$. Consequently, $g = g_2 \circ g_1$ and $h = h_2 \circ h_1$ form a pair of embeddings $g, h \colon  \A^-(b) \to \A^-(x,y)$ satisfying 
 \[
 g {\upharpoonright}_A = h {\upharpoonright}_A \, \, \text{ and } \, \, g(b) = x \neq y = h(b).
 \] 
 As $\A^-(x, y)$ is a field, the above display yields $b \notin \mathsf{d}_\mathsf{RCR}(\A^-, \A^-(b))$.
\end{proof}

    By \cref{Claim : very last argument of the main result} there exists a pair of homomorphisms $g,h \colon \B^- \to \D$ with $\D \in \mathsf{RCR}$ such that $g {\upharpoonright}_A = h {\upharpoonright}_A$ and $g(b) \neq h(b)$. In view of \cref{Cor : subreducts}\eqref{item : subreducts : 2}, we may assume that $\D = \boldsymbol{E}^-$ for some $\E \in \mathsf{PM}$. Therefore, we can apply \cref{Cor : RCR sub are ICM sub}, obtaining that $g$ and $h$ can be viewed as  homomorphisms $g, h \colon \B \to \E$ of perfect meadows. Since $g {\upharpoonright}_A = h {\upharpoonright}_A$ and $g(b) \neq h(b)$, we conclude that  $b \notin \mathsf{d}_{\mathsf{PM}}(\A, \B)$.
    \end{proof}

 We have thus established that adding the implicit operations of “taking weak inverses” and “taking weak prime roots” to the reduced commutative rings produces the canonical completion of this class in which every monomorphism (and every epimorphism) becomes regular. This completion is equationally axiomatizable (see \cref{Thm : ICM axioms}) and has the amalgamation property (see \cref{Thm : ICM has AP}). 
We close this section by highlighting that it also forms a discriminator variety.

\begin{Remark}\label{Rem : discriminator}
A variety  is \emph{discriminator} when it is of the form $\VVV(\K)$ for some class  $\K$ with a term  $t(x, y, z)$ such that for all $\A \in \K$ and $a, b, c \in A$,
\begin{equation}\label{Eq : discriminator term}
    t^\A(a, b, c) = \begin{cases*}
                    c & if  $a = b$;  \\
                     a & otherwise
                 \end{cases*}
\end{equation}
(see, e.g., \cite[Sec.\ IV.9]{BuSa00}). The importance of discriminator varieties derives from the fact that they have many desirable properties such as a representation theorem in terms of Boolean products with subdirectly irreducible factors (see, e.g., \cite[Thm.\ IV.9.4]{BuSa00}) or the fact that they are \emph{arithmetical}, in the sense that they satisfy a general form of the Chinese Remainder Theorem (see, e.g., \cite[p.~35 \& Thm.~2.2.1]{KP01}).

    Recall that $\mathsf{PM} = \III\SSS\PPP(\mathsf{EPF})$ by definition and that $\mathsf{PM}$ is a variety by \cref{Thm : ICM axioms}. Therefore, $\mathsf{PM} = \VVV(\mathsf{EPF})$. Furthermore, it is straightforward to check that the term
\[
t(x, y, z) = x(x-y)(x-y)^{*}+z(1-(x-y)(x-y)^{*})
\]
satisfies \eqref{Eq : discriminator term} for every $\A \in \mathsf{EPF}$ (see \cite[p.\ 300]{MR2793402}).\ Consequently, $\mathsf{PM}$ is a discriminator variety.
\qed
    \end{Remark}

    \section{Dominions}

 We close this paper with a  description of dominions in every 
class $\K$ such that $\mathsf{F} \subseteq \K \subseteq \mathsf{RCR}$. 

\begin{Definition}
Let $\A \in \mathsf{RCR}$. The set of prime ideals of $\A$ will be denoted by $\mathsf{Spec}(\A)$.
\end{Definition}

Let $\A \in \mathsf{RCR}$ and $I \in \mathsf{Spec}(\A)$. We recall  that  $\mathsf{frac}(\A / I)$ is a field (see \cref{Prop : prime ideals and integral domains}) and that $\mathsf{acl}(\mathsf{frac}(\A / I))^+$ is an expanded perfect field by \cref{Prop : acl(A) in ICF}. 

\begin{Definition}
Let $\A \in \mathsf{RCR}$. Then
\benroman 
\item for every $I \in \mathsf{Spec}(\A)$ let 
\[
\mathsf{icf}_I(\A) = \mathsf{acl}(\mathsf{frac}(\A / I))^+;
\]
\item for all $I, J \in \mathsf{Spec}(\A)$ let
\[
\mathsf{icf}_{I \times J}(\A) = \mathsf{icf}_I(\A) \times \mathsf{icf}_J(\A).
\]
\eroman
\end{Definition}

Observe that for all $\A \in \mathsf{RCR}$ and $I \in \mathsf{Spec}(\A)$,
\[
\A / I \leq \mathsf{acl}(\mathsf{frac}(\A / I)) =   \mathsf{icf}_I(\A)^-.
\] 
Our aim is to prove the following.

\begin{Theorem}\label{Thm : dominions}
Let $\K$ be a 
class such that $\mathsf{F} \subseteq \K \subseteq \mathsf{RCR}$. Moreover, let $\A \leq \B \in \K$ and $b \in B$. Then $b \in \mathsf{d}_\K(\A, \B)$ if and only if for all $I, J \in \mathsf{Spec}(\B)$ we have
\[
\langle b + I, b + J \rangle \in \mathsf{Sg}^{\mathsf{icf}_{I \times J}(\B)} ( \{ \langle a + I, a + J \rangle : a \in A \}).
\]
\end{Theorem}

Before proving \cref{Thm : dominions}, we show that this result admits the following improvement in the case where $\mathsf{K} = \mathsf{F}$. Recall that $\mathsf{acl}(\A)^+$ is an expanded perfect field for every field $\A$ (see \cref{Prop : acl(A) in ICF}).

\begin{Corollary}\label{Cor : dominions in fields}
        Let $\A \leq \B \in \mathsf{F}$. Then 
        \[
        \mathsf{d}_{\mathsf{F}}(\A, \B) = B \cap \mathsf{Sg}^{\mathsf{acl}(\B)^+}(A).
        \]
    \end{Corollary}
    
    \begin{proof}
We will use repeatedly the assumption that $\B$ is a field. First, observe that the map $ h \colon \B \to \mathsf{frac}(\B / \{ 0 \})$ defined as $h(b) = b / \{ 0 \}$ for every $b \in B$ is an isomorphism by the definitions of a quotient ring and a fraction field. Together with the definition of $\mathsf{icf}_{\{0 \}}(\B)$, this ensures the existence of an isomorphism $g \colon \mathsf{acl}(\B)^+ \to \mathsf{icf}_{\{0 \}}(\B)$ such that $g(b) = b / \{ 0 \}$ for every $b \in B$. Next, observe that the only prime ideal of $\B$ is $\{ 0 \}$. Therefore,  \cref{Thm : dominions} translates to the following: for every $b \in B$,
\[
b \in \mathsf{d}_\mathsf{F}(\A, \B) \Longleftrightarrow \langle b, b \rangle \in \mathsf{Sg}^{\mathsf{acl}(\B)^+ \times \mathsf{acl}(\B)^+} (\{ \langle a, a \rangle : a \in A \}).
\] 
As the right hand side of the above display is equivalent to $b \in \mathsf{Sg}^{\mathsf{acl}(\B)^+}(A)$, we are done.
    \end{proof}

The proof of \cref{Thm : dominions} hinges upon the next observation. 

\begin{Proposition} \label{Prop : homs lift}
Let $h \colon \A \to \B$ be a ring homomorphism with $\A \in \mathsf{RCR}$ and $\B \in \mathsf{F}$.
   Then there exist $I \in \mathsf{Spec}(\A)$ and a homomorphism $g \colon \mathsf{icf}_{I}(\A) \to \mathsf{acl}(\B)^+$ of perfect meadows such that $g(a+I) = h(a)$ for every $a \in A$.
\end{Proposition}

\begin{proof}
Since $h \colon \A \to \B$ is a homomorphism, we have $h[\A] \leq \B$. As $\B$ is a field, this yields that $h[\A]$ is an integral domain. Consequently, the kernel $I$ of $h$ is a prime ideal of $\A$ (see \cref{Prop : prime ideals and integral domains}). Moreover, the map $i \colon \A / I \to \B$ defined for every $a \in A$ as $i(a + I) = h(a)$ is a ring embedding.
Applying \cref{Prop : embedding fields of fractions}, we obtain a ring embedding $j \colon \mathsf{frac}(\A/I) \to \B$ such that 
\begin{equation} \label{eq : h eq}
    j(a+I) = h(a) \text{ for every } a \in A.
\end{equation}
By \cref{Prop : extensing acl embeddings} the map $j$ extends to a ring embedding $g \colon \mathsf{acl}(\mathsf{frac}(\A/I)) \to \mathsf{acl}(\B)$. Since $\mathsf{acl}(\mathsf{frac}(\A/I))^+, \mathsf{acl}(\B)^+ \in \mathsf{PM}$ by \cref{Prop : acl(A) in ICF}, from \cref{Cor : RCR sub are ICM sub} it follows that $g$ can be viewed as a homomorphism $g \colon \mathsf{acl}(\mathsf{frac}(\A/I))^+ \to \mathsf{acl}(\B)^+$ of perfect meadows.
As $\mathsf{icf}_I(\A) = \mathsf{acl}(\mathsf{frac}(\A))^+$ by definition and $g$ extends $j$, the validity of \eqref{eq : h eq} concludes the proof.
\end{proof}

We are now ready to prove \cref{Thm : dominions}.

\begin{proof}
We begin by proving the implication from left to right. To this end, we reason by contraposition. Suppose that there exist $I, J \in \mathsf{Spec}(\B)$ such that $\langle b+I, b+J \rangle \notin C$, where 
\[
\C = \mathsf{Sg}^{\mathsf{icf}_{I \times J}(\B)} (\{\langle a+I, a+J \rangle : a \in A\}).
\]
 We need to show that $b \notin \mathsf{d}_\K(\A, \B)$. Notice that $\langle b+I, b+J \rangle \in \mathsf{icf}_{I \times J}(\B)$ and $\C \leq \mathsf{icf}_{I \times J}(\B) \in \mathsf{PM}$. Since monomorphisms are regular in $\mathsf{PM}$ by \cref{Thm : ICM has SES}, there exist $\D \in \mathsf{PM}$ and a pair of homomorphisms $g,h \colon \mathsf{icf}_{I \times J}(\B) \to \D$ of perfect meadows such that 
\begin{equation} \label{Eq : dom witnesses}
    g {\upharpoonright}_C = h {\upharpoonright}_C \quad \text{and} \quad g(\langle b+I, b+J \rangle) \neq h(\langle b+I, b+J \rangle).
\end{equation}

\begin{Claim} \label{Claim : C' in F}
    We may assume that $\D \in \mathsf{EPF}$.
\end{Claim}
\begin{proof}[Proof of the Claim]
Recall that $\mathsf{PM} = \III\SSS\PPP(\mathsf{EPF})$ by definition.\ 
Therefore, there exists an embedding $f \colon \D \to \prod_{k \in K} \D_k$, where $\D_k \in  \mathsf{EPF}$ for each $k \in K$. 
From \eqref{Eq : dom witnesses} and the assumption that  $f$  is an embedding it follows that there exists $j  \in K$ such that 
\[
(p_j \circ f \circ g){\upharpoonright}_C = (p_j \circ f \circ h){\upharpoonright}_C \, \, \text{ and }\, \, (p_j \circ f \circ g)(\langle b + I, b + J \rangle) \ne (p_j \circ f \circ h)(\langle b + I, b + J \rangle),
\]
where $p_j \colon \prod_{k \in K}\D_k \to \D_j$ is the projection onto the $j$-th coordinate.
Therefore, we may replace $\D$ by  $\D_j$ in the main proof. As $\D_j  \in \mathsf{EPF}$, we are done.
\end{proof}

In view of \cref{Claim : C' in F}, we have $\D^- \in \mathsf{F}$. As $\mathsf{F} \subseteq \K$ by assumption, we obtain $\D^- \in \K$. Moreover, observe that $g$ and $h$ can be viewed as ring homomorphisms $g, h \colon \mathsf{icf}_{I \times J}(\B)^- \to \D^-$. 
Lastly, let $f \colon \B \to  \mathsf{icf}_{I \times J}(\B)^-$ be the ring homomorphism defined as $f(a) = \langle a+I, a+J \rangle$ for every $a \in B$. Therefore, $g \circ f, h \circ f \colon \B \to \D^-$ are ring homomorphisms with $\D^- \in \K$. Since the definitions of $f$ and $C$ ensure that $f[A] \subseteq C$, the left hand side of \eqref{Eq : dom witnesses} yields $(g \circ f) {\upharpoonright}_A = (h \circ f) {\upharpoonright}_A$. On the other hand, the definition of $f$ and the right hand side of \eqref{Eq : dom witnesses} guarantee that
\[
g(f(b)) = g(\langle b + I, b + J \rangle) \ne h(\langle b + I, b + J \rangle) = h(f(b)).
\]
Hence, we conclude that $b \notin \mathsf{d}_\K(\A, \B)$, as desired.

Next, we prove the implication from right to left. Suppose, with a view to contradiction, that $\langle b + I, b + J \rangle \in \mathsf{Sg}^{\mathsf{icf}_{I \times J}(\B)} (\{\langle a+I, a+J \rangle : a \in A\})$ for all $I, J \in \mathsf{Spec}(\B)$ and that $b \notin \mathsf{d}_\K(\A, \B)$. The latter implies that there exists a pair of ring homomorphisms $g,h \colon \B \to \C$ with $\C \in \K$
such that 
\begin{equation} \label{Eq : g,h assumption main}
    g {\upharpoonright}_A = h {\upharpoonright}_A \, \, \text{ and }\, \,  g(b) \neq h(b).
\end{equation}

An argument analogous to the one detailed in the proof of \cref{Claim : C' in F} (with the only difference that  the role of the equality $\mathsf{PM} = \III\SSS\PPP(\mathsf{EPF})$ in the proof  should now be played by $\mathsf{RCR} = \III\SSS\PPP(\mathsf{F})$)
shows that $\C$ may be chosen in $\mathsf{F}$.
Consequently, we can apply \cref{Prop : homs lift}, obtaining $I, J \in \mathsf{Spec}(\B)$ and a pair of homomorphisms of perfect meadows $g^* \colon \mathsf{icf}_I(\B) \to \mathsf{acl}(\C)^+$ and $h^* \colon \mathsf{icf}_J(\B) \to \mathsf{acl}(\C)^+$ such that for every $c \in B$,
\begin{equation}\label{Eq : what the star does to a map}
    g^*(c + I) = g(c) \, \, \text{ and } \, \, h^*(c + J) = h(c).
\end{equation}

As $\langle b+I, b+J \rangle \in \mathsf{Sg}^{\mathsf{icf}_{I \times J}(\B)} (\{\langle a+I, a+J \rangle : a \in A\})$ by assumption, there exist $a_1, \dots, a_n \in A$ and a term $t(x_1, \dots, x_n)$ of $\mathsf{PM}$ such that
    \begin{equation} \label{Eq : defining term}
    \langle b+I, b+J \rangle  = t^{\mathsf{icf}_{I \times J}(\B)}(\langle a_1+I, a_1+J \rangle, \dots, \langle a_n+I, a_n+J \rangle).
    \end{equation}
Recall that $\C \leq  \mathsf{acl}(\C)$ by the definition of an algebraic closure. Therefore, $g[B] \cup h[B] \subseteq C \subseteq \mathsf{acl}(C)^+$.
We will show that
\begin{equation}\label{Eq : yet another display in the last proof}
    g(b) = t^{\mathsf{acl}(\C)^+}(g(a_1), \dots, g(a_n)) \, \, \text{ and }\, \, h(b) = t^{\mathsf{acl}(\C)^+}(h(a_1), \dots, h(a_n)).
\end{equation}
By symmetry it suffices to prove the left hand side of the above display. To this end, observe that
\begin{align*}
    g(b) &= g^*(b + I) = g^*(t^{\mathsf{icf}_I(\B)}(a_1 + I, \dots, a_n + I))\\
    &=t^{\mathsf{acl}(\C)^+}(g^*(a_1 + I), \dots, g^*(a_n + I)) =  t^{\mathsf{acl}(\C)^+}(g(a_1), \dots, g(a_n)),
\end{align*}
where the first and the last equalities hold by the left hand side of \eqref{Eq : what the star does to a map}, the second by \eqref{Eq : defining term} and the definition of $\mathsf{icf}_{I \times J}(\B)$, and the third because $g^* \colon \mathsf{icf}_{I}(\B) \to \mathsf{acl}(\C)^+$ is a homomorphism of perfect meadows and $t(x_1, \dots, x_n)$ a term of $\mathsf{PM}$. This establishes \eqref{Eq : yet another display in the last proof}. Lastly, from \eqref{Eq : yet another display in the last proof}, the left hand side of 
\eqref{Eq : g,h assumption main}, and the fact that $a_1, \dots, a_n \in A$ it follows that
\[
g(b) = t^{\mathsf{acl}(\C)^+}(g(a_1), \dots, g(a_n))  = t^{\mathsf{acl}(\C)^+}(h(a_1), \dots, h(a_n)) = h(b),
\]
a contradiction with the right hand side of \eqref{Eq : g,h assumption main}.
    \end{proof}

\subsection*{AI disclosure} This paper was  produced without the help of AI, as witnessed by very similar   arXiv versions which date back to a time preceding the surge of AI in mathematics. Although this version contains a few minor modifications in relation to the arXiv preprints, we opted to keep the paper a purely human work.

\subsection*{Acknowledgements} We are grateful to Luca Reggio and Ivan Di Liberti for suggesting a better terminology for the paper.

\end{document}